\documentclass[11pt,reqno]{amsart}
\usepackage{amsmath, amssymb, amsthm}
\usepackage{url}
\usepackage[breaklinks]{hyperref}

\renewcommand{\k}{\kappa}

\renewcommand{\(}{\left\(}
\renewcommand{\)}{\right\)}
\renewcommand{\[}{\left\[}
\renewcommand{\]}{\right\]}
\renewcommand{\i}{\infty}
\numberwithin{equation}{section}
 \theoremstyle{plain}
\newtheorem{theorem}{Theorem}[section]
\newtheorem{lemma}[theorem]{Lemma}
\newtheorem{remark}[]{Remark}

\newtheorem{corollary}[theorem]{Corollary}

\newtheorem{example}[]{Example}

   \makeatletter
\def\proof{\@ifnextchar[{\@oproof}{\@nproof}}
\def\@oproof[#1][#2]{\trivlist\item[\hskip\labelsep\textit{#2 Proof of\
#1.}~]\ignorespaces}
\def\@nproof{\trivlist\item[\hskip\labelsep\textit{Proof.}~]\ignorespaces}

\makeatother

\begin{document}
\title[Partition implications of a new three parameter $q$-series identity]{Partition implications of a new three parameter $q$-series identity} 

\author{Atul Dixit and Bibekananda Maji}
\address{Discipline of Mathematics, Indian Institute of Technology Gandhinagar, Palaj, Gandhinagar 382355, Gujarat, India} 
\email{adixit@iitgn.ac.in, bibekananda.maji@iitgn.ac.in}
\begin{abstract}
A generalization of a beautiful $q$-series identity found in the unorganized portion of Ramanujan's second and third notebooks is obtained. As a consequence, we derive a new three-parameter identity which is a rich source of partition-theoretic information. In particular, we use this identity to obtain a generalization of a recent result of Andrews, Garvan and Liang, which itself generalizes the famous result of Fokkink, Fokkink and Wang. This three-parameter identity also leads to several new weighted partition identities as well as a natural proof of a recent result of Garvan. This natural proof gives interesting number-theoretic information along the way. We also obtain a new result consisting of an infinite series involving a special case of Fine's function $F(a,b;t)$, namely, $F(0,q^n;cq^n)$. For $c=1$, this gives Andrews' famous identity for $\textup{spt}(n)$ whereas for $c=-1, 0$ and $q$, it unravels new relations that the divisor function $d(n)$ has with other partition-theoretic functions such as the largest parts function $\textup{lpt}(n)$. 
\end{abstract}
\maketitle
\section{Introduction}\label{intro}
The unorganized portion of Ramanujan's second and third notebooks contains five $q$-series identities \cite[p.~354--355]{ramanujanoriginalnotebook2}, \cite[p. 302--303]{ramanujantifr} that were first proved by Berndt \cite[p.~262--265]{bcbramforthnote}. The first one is \cite[p.~262, Entry 1]{bcbramforthnote}
\begin{equation*}
\frac{(-aq)_{\infty}}{(bq)_{\infty}}=\sum_{n=0}^{\infty}\frac{(-b/a)_na^nq^{n(n+1)/2}}{(q)_n(bq)_n},
\end{equation*}
where $a\neq 0, 1-bq^n\neq 0, n\geq 1$. It is also recorded in the Lost Notebook \cite[p.~370]{lnb}. Here, and throughout the paper, $|q|<1$, and
\begin{align}
(A)_0 &:=(A;q)_0 =1, \qquad \nonumber\\
(A)_n &:=(A;q)_n  = (1-A)(1-Aq)\cdots(1-Aq^{n-1}),
\qquad n \geq 1, \label{finpoch}\\
(A)_{\infty} &:=(A;q)_{\i}  = \lim_{n\to\i}(A;q)_n\nonumber.
\end{align}
Moreover, the definition in \eqref{finpoch} can be extended to all integers $n$ by defining 
\begin{equation*}
(A)_n=\frac{(A)_{\infty}}{(Aq^{n})_{\infty}}.
\end{equation*}
The third one in the list \cite[p.~354]{ramanujanoriginalnotebook2}, \cite[p.~263, Entry 2]{bcbramforthnote} states that for $a\neq q^{-n}, n\geq 1$,
\begin{equation*}
(aq)_{\infty}\sum_{n=1}^{\infty}\frac{na^nq^{n^2}}{(q)_n(aq)_n}=\sum_{n=1}^{\infty}\frac{(-1)^{n-1}a^nq^{n(n+1)/2}}{1-q^n}.
\end{equation*}
The sequence generated by the special case $a=1$ of the series on the above left-hand side was posted by Deutsch on the Online Encyclopedia of Integer Sequences (A115995) and Jovovic found that the same sequence is generated by the special case $a=1$ of the above right-hand side. Andrews, Chan and Kim \cite[p.~82]{ack} rediscovered the above identity. The second one in Ramanujan's list \cite[p.~354]{ramanujanoriginalnotebook2}, \cite[p.~264, Entry 4]{bcbramforthnote} is given by
\begin{equation}\label{Garvan's identity}
\sum_{n=1}^{\infty} \frac{(-1)^{n-1} z^n q^{\frac{n(n+1)}{2} } }{(1-q^n) (z q)_n  } = \sum_{n=1}^{\infty} \frac{z^n q^n }{1-q^n}
\end{equation}
where $z\neq q^{-n}, n\geq 1$, was rediscovered by Uchimura \cite[Equation (3)]{uchimura81} and Garvan \cite{garvan0}. The special case $z=1$ of the above identity is due to Kluyver \cite{kluyver}:
\begin{align}\label{Uchimura}
\sum_{n=1}^{\infty} \frac{(-1)^{n-1} q^{\frac{n(n+1)}{2} } }{(1-q^n) ( q)_n  } = \sum_{n=1}^{\infty} \frac{ q^n }{1-q^n},
\end{align}
It was rediscovered by Fine \cite[p.~14, Equations (12.4), (12.42)]{fine} and Uchimura \cite[Theorem 2]{uchimura81}. Kluyver's identity admits a beautiful partition-theoretic interpretation due to Fokkink, Fokkink and Wang \cite{ffw95}. Before stating this interpretation, we discuss the notation used throughout the paper.
\begin{itemize}
\item $\pi$: an integer partition,

\item $p(n)$: the number of integer partitions of $n$

\item $ s(\pi):=$ the smallest part of $\pi$,

\item $l(\pi):=$ the largest part of $\pi$,

\item $\#(\pi):=$ the number of parts of $\pi$,

\item $\mathrm{rank}(\pi)= l(\pi)- \#(\pi)$,

\item $L(\pi):=$ total number of appearances of the largest part of $\pi$,
  
\item $\nu_{d}(\pi):=$ the number of parts of $\pi$ without multiplicity,

\item $\mathcal{P}(n):=$ collection of all integer partitions of $n$,

\item $\mathcal{D}(n):=$ collection of all distinct partitions of $n$,

\item $\mathcal{P}_{o}(n):=$ collection of all overpartitions of $n$,

\item $\mathcal{P}^{*}(n):=$ partitions without gaps (that is, partitions into consecutive integers with smallest part $1$).
\end{itemize}

Then the result of Fokkink, Fokkink and Wang states that if $d(n)$ denotes the number of divisors of $n$,
\begin{align}\label{ffwidty}
\sum_{ \pi \in \mathcal{D}(n)  } (-1)^{ \# (\pi)-1 } s(\pi)=d(n).
\end{align}
In his seminal paper \cite{andrews08}, Andrews revisited the above interpretation and showed that its generating function version, namely Kluyver's identity \eqref{Uchimura}, is simply a corollary of the differentiation of the $q$-analogue of Gauss' theorem \cite[p.~20, Corollary 2.4]{gea1998}. There is a vast literature on such $q$-series identities related to divisor functions.
For developments on this topic since the appearance of Kluyver's identity \eqref{Uchimura}, we refer the reader to a paper of Guo and Zeng \cite{guozeng2015} though we do point out those that are relevant here. As shown in the paper of Ismail and Stanton \cite{ismailstantonanncomb}, the genesis of many of such $q$-series identities lies in the theory of basic hypergeometric functions. For recent developments since the appearance of Guo and Zeng's paper, see \cite{xu} and \cite{xucen}.

Andrews, Garvan and Liang \cite{agl13} denoted the left-hand side of \eqref{ffwidty} by $\textup{FFW}(1,n)$ and considered its generalization $\textup{FFW}(c,n)$ defined by
\begin{equation}\label{ffwcn}
\textup{FFW}(c,n):=\sum_{\pi\in\mathcal{D}(n)}(-1)^{\#(\pi)-1}\left(1+c+\cdots+c^{s(\pi)-1}\right).
\end{equation}
They showed that \cite[Theorem 3.5]{agl13}
\begin{align}\label{ffwz}
\sum_{n=1}^{\infty} \mathrm{FFW}(c,n) q^n = \sum_{n=1}^{\infty} \frac{(-1)^{n-1} q^{\frac{n(n+1)}{2}}}{(1-cq^n) (q)_n } = \frac{1}{1-c} \left(1- \frac{(q)_{\infty} }{(c q)_\infty }  \right).
\end{align}
The above identity is valid for $|q|<|c|<1$. Another generalization of \eqref{ffwidty} was obtained by Patkowski through a `sum of tails' identity \cite[Corollary 2.4]{patkowskicm}.

In this paper, among other things, we generalize the above result of Andrews, Garvan and Liang by introducing one more variable, namely, $z$. This is achieved by first obtaining a more general result which generalizes a yet another beautiful identity of Ramanujan which is the last one in \cite[p.~354]{ramanujanoriginalnotebook2}, \cite[p.~263, Entry 3]{bcbramforthnote}. This identity does not seem to have received the attention it should have. It states that for $a\neq 0$ and $(b)_n\neq 0$ for $n\geq 1$,
\begin{align}\label{neglected}
\sum_{n=1}^{\infty} \frac{ (b/a)_n a^n }{ (1- q^n) (b)_n } = \sum_{n=1}^{\infty} \frac{a^n - b^n }{ 1- q^n }. 
\end{align}
\section{Main results}\label{mare}
Our generalization of \eqref{neglected} is given in the following theorem.
\begin{theorem}\label{gen of Ramanujan's identity}
 Let $a, b, c$ be three complex numbers such that $|a|<1$, $|b|<1$ and $|c|\leq 1$. Then
\begin{align}\label{entry3gen}
\sum_{n =1}^{\infty}  \frac{ (b/a)_n a^n }{ (1- c q^n) (b)_n } =  \sum_{m=0}^{\infty}\frac{(b/c)_mc^m}{(b)_m}\left(\frac{aq^m}{1-aq^m}-\frac{bq^m}{1-bq^m}\right).
\end{align}
In addition, if $|b|<|c|$,
\begin{align}\label{entry3gen00}
\sum_{n =1}^{\infty}  \frac{ (b/a)_n a^n }{ (1- c q^n) (b)_n }=\frac{ ( b/c )_{\infty }}{(b)_{\infty} }   \sum_{n=0}^{\infty} \frac{(c)_n ( b/c)^n }{(q)_n}    \sum_{m=1}^{\infty} \frac{a^m - b^m }{1- c q^{m+n}  }.
\end{align}
\end{theorem}
It is easy to see that when we let $c=1$ in \eqref{entry3gen00}, only the $n=0$ term survives in the series on the right-hand side resulting in \eqref{neglected}.

Letting $a\to 0$ and replacing $b$ by $zq$ in \eqref{entry3gen} leads to the aforementioned generalization of Andrews, Garvan and Liang's \eqref{ffwz} given below. 
\begin{theorem}\label{gGaravan}
For $|zq|<1$ and $|c|\leq 1$,
\begin{align}\label{gen of Garavan}
\sum_{n=1}^{\infty} \frac{(-1)^{n-1} z^n q^{\frac{n(n+1)}{2} } }{(1-c q^n) (z q)_n  } = \frac{z}{c}\sum_{n=1}^{\infty}\frac{(zq/c)_{n-1}}{(zq)_n}(cq)^n.
\end{align}
\end{theorem}
The study of weighted counts of partitions was initiated by Alladi \cite{alladiantheini}, \cite{alladitrans}, \cite{alladiramer}. There have been interesting further studies on weighted partition identities \cite{alladi2016}, \cite{andrews08}, \cite{berkovichuncujnt}, \cite{garvan1}, to name a few. Comparing the coefficients of $q^n$ on both sides of \eqref{gen of Garavan}, we obtain the following general weighted partition identity which gives several interesting corollaries, some of which are new and others, well-known. 
\begin{theorem}\label{gwpi}
If $z$ and $c$ are not functions of $q$,
\begin{align}\label{gwpieqn}
 \sum_{\pi \in \mathcal{D}(n) } (-1)^{\#(\pi)-1} z^{l(\pi) + 1 -s(\pi)} c^{s(\pi)-1}\frac{\left(\frac{z}{c} \right)^{s(\pi)} - 1}{ \left(\frac{z}{c} \right)-1}=\sum_{\pi \in \mathcal{P}(n) } z^{\#(\pi)} c^{l(\pi) -1} \left(1- \tfrac{1}{c}  \right)^{\nu_d(\pi) -1 }.
\end{align}
\end{theorem}
If we let $z=1$ in Theorem \ref{gGaravan}, we obtain the following result, of which \eqref{fcz1} as well as the left-hand side of \eqref{fcz1pi} was obtained by Andrews, Garvan and Liang.
\begin{corollary}\label{fixed c and z tends 1}
For $|c|\leq 1$, we have
\begin{align}\label{fcz1}
\sum_{n=1}^{\infty} \frac{(-1)^{n-1} q^{\frac{n(n+1)}{2} } }{(1-c q^n) ( q)_n  }= \sum_{n=1}^{\infty}\frac{c^{n-1}(q/c)_{n-1} q^n}{(q)_n}=\frac{1}{1-c} \left(1- \frac{(q)_{\infty} }{(c q)_\infty }  \right).
\end{align}
Hence if $c$ is not a function of $q$,
\begin{align}\label{fcz1pi}
\sum_{ \pi \in \mathcal{D}(n)  } (-1)^{\#(\pi) -1} \left(\frac{c^{s(\pi)}-1}{c-1}\right) =  \sum_{ \pi \in \mathcal{P}(n)} c^{l(\pi)-1} \left( 1- \frac{1}{c} \right)^{\nu_{d}(\pi) -1}.
\end{align}
\end{corollary}
One of the most important $q$-series of Ramanujan, and which has been the source of investigation from the point of view of both analytic and algebraic number theory, is
\begin{equation*}
\sigma(q):=\sum_{n=0}^{\infty} \frac{q^{\frac{n(n+1)}{2}}}{(-q)_n}.
\end{equation*}
For deep results associated with $\sigma(q)$, the reader is referred to \cite{andrews1986}, \cite{andrewsmonthly86} and \cite{adh}. In \cite{bandix1}, some results associated with one- and two-variable generalization of $\sigma(q)$, namely $\sigma(c, q)$ and $\sigma(c, d, q)$ were obtained. These functions are defined by
\begin{align}
\sigma(c, q)&:=\sum_{n=0}^{\infty} \frac{q^{n(n+1)/2} }{(1-c q^n ) (-q)_n},\label{otvg}\\
\sigma(c, d, q)&:=\sum_{n=0}^{\infty}\frac{(-cd)_n(1-cdq^{2n})q^{n(n+1)/2}}{(1-cq^n)(1-dq^n)(-q)_n}.\label{otvg1}
\end{align}
For example, it was shown that \cite[Theorem 1.1]{bandix1} for $|c|<1$,
\begin{equation*}
\sum_{n=1}^{\infty}\frac{q^{n(n+1)/2}}{(1-cq^n)(-q)_n}=\frac{1}{(-c)_{\infty}}\left(\sigma(q)+\frac{(-c)_{\infty}}{c-1}+2\sum_{m,n=0}^{\infty}\frac{(-q)_m}{(q)_m(q)_n}\frac{(-1)^nq^{n(n+1)/2}c^{m+n+1}}{(1-q^{n+m+1})}\right).
\end{equation*}
Letting $z=-1$ in Theorem \ref{gGaravan} gives us a simpler representation of $\sigma(c, q)$ given below.
\begin{corollary}\label{gen of sigma(q)}
For $|c|\leq 1$, we have
\begin{align*}
\sum_{n=1}^{\infty} \frac{q^{n(n+1)/2} }{(1-c q^n ) (-q)_n} = \sum_{n=1}^{\infty}  \frac{\left( \frac{-q}{c} \right)_{n-1} c^{n-1}q^n}{(-q)_{n}}.
\end{align*}
If $c$ is not a function of $q$, then 
\begin{align*}
\sum_{ \pi \in \mathcal{D}(n)  } (-1)^{\mathrm{rank}(\pi)} \frac{ (1-(-c)^{s(\pi)})}{1+c} =  \sum_{ \pi \in \mathcal{P}(n)} (-1)^{\#(\pi)-1} c^{l(\pi)-1} \left( 1- \frac{1}{c} \right)^{\nu_{d}(\pi) -1}.
\end{align*}
\end{corollary}
When $c=0$, the left-hand side of the above equation readily implies the result in \cite{andrews1986} that $\sigma(q)$ is the generating function for the excess number of partition of $n$ into distinct parts with even rank over those with odd rank. 

Now let $c \rightarrow 1$ in Theorem \ref{gGaravan} to get 
\begin{corollary}\label{fixed z and c tends to 1}
Equation \ref{Garvan's identity} holds. Moreover, if $z$ is not a function of $q$,
\begin{align*}
\sum_{ \pi \in \mathcal{D}(n)  } (-1)^{\#(\pi) -1}  z^{l(\pi)+1}\left(\frac{1- z^{-s(\pi)}}{z-1}\right) = \sum_{d |n }z^d .
\end{align*}
\end{corollary}
Again, if we let $c \rightarrow -1$ in Theorem \ref{gGaravan}, then we get
\begin{corollary}\label{fixed z and c tends to -1}
For $|z q| <1$, we have
\begin{align*}
\sum_{n=1}^{\infty} \frac{(-1)^{n-1} z^n q^{\frac{n(n+1)}{2} } }{(1 + q^n) (z q)_n  } = \sum_{n=1}^\infty  \frac{(-1)^{n-1}(-z q)_{n-1}  q^n}{ (zq)_n }. 
\end{align*}
Hence for $z$ is not a function of $q$,
\begin{align*}
\sum_{ \pi \in \mathcal{D}(n)  } (-1)^{\#(\pi)+s(\pi)}  z^{l(\pi)+1-s(\pi) }\left(\frac{ (-z)^{s(\pi)} -1}{-z-1} \right) = \sum_{ \pi \in \mathcal{P}(n)} (-1)^{l(\pi)-1} z^{\#(\pi)}  2^{\nu_d(\pi) -1 }.
\end{align*}
\end{corollary}
Few more weighted partition identities resulting from Theorem \ref{gwpi} are discussed in Section \ref{weight}.

Theorem \ref{gGaravan} also allows us to give a natural proof of a recent identity of Garvan \cite[Equation (1.3)]{garvan1} used to prove nice weighted partition identities \cite[Corollary 1.3]{garvan1}. It states that for $|z|\leq 1$,
\begin{align}\label{garvan's 2nd thm}
\sum_{n=1}^{\infty} \frac{(-1)^{n-1} z^n q^{n^2}}{(zq;q^2)_n (1- z q^{2n}) } = \sum_{n=1}^{\infty} \frac{z^n q^{ \frac{n(n+1)}{2}} (q;q)_{n-1} }{ (zq; q)_n  }.
\end{align}
Garvan's proof in \cite{garvan1} requires that both sides of the identity be known first, that is, the identity is proved by reducing it to an equivalent one and the latter is proved by showing that the coefficients of $z^m$ of this equivalent identity are equal. Our proof, on the other hand, is direct.

Differentiation of $q$-series identities often plays a crucial role in deriving new and fundamental results in the theory of partitions and basic hypergeometric series \cite{andrews08}, \cite{agl13}, \cite{cdg}, \cite{chern}, \cite{dixityee}. The $q$-series identities in this paper are no exception to this. 

Differentiating both sides of \eqref{gen of Garavan} with respect to $z$ and letting $z=1$ leads to the following new result.
\begin{theorem}\label{genc}
Let Fine's function $F(a,b;t)$ be defined by \cite[p.~1]{fine}
\begin{equation*}
F(a,b;t):=\sum_{n=0}^{\infty}\frac{(aq)_n}{(bq)_n}t^n.
\end{equation*}
Then for $|c|<1$, we have
\begin{align}\label{finefeqn}
&\sum_{n=1}^{\infty}\frac{(-1)^{n-1}nq^{n(n+1)/2}}{(1-cq^n)(q)_n}+\sum_{n=1}^{\infty}\frac{q^{n(n+1)}(q^{n+1})_{\infty}}{(1-q^n)(q)_n}F(0,q^n;cq^n)\nonumber\\
&=\frac{-c}{(1-c)^2}+\frac{(q)_{\infty}}{(c)_{\infty}}\left(\frac{c}{1-c}+\sum_{n=1}^{\infty}\frac{(cq)_n}{(q)_n}\frac{q^n}{1-q^n}\right).
\end{align}
\end{theorem}
While we are unable to further simplify the second series on the left-hand side of the above theorem for a general $c$, the simplification is achieved in the special cases $c=0, \pm 1$. In fact, it is a pleasant surprise that letting $c\to 1$ in the above result and equating the coefficients of $q^n$ on both sides of the resulting identity, yields Andrews' famous identity \cite[Theorem 3]{andrews08} given below. This is proved in Section \ref{diffe}.
\begin{corollary}\label{spt}
Let $\textup{spt}(n)$ enumerate the number of smallest parts in all partitions of $n$ and $N_{2}(n)$ be the second Atkin-Garvan rank moment \cite[Equation (2.13)]{andrews08}. Then
\begin{equation}\label{sptidentity}
\textup{spt}(n)=np(n)-\frac{1}{2}N_{2}(n).
\end{equation}
\end{corollary}
While the case $c\to 0$ of Theorem \ref{gGaravan} is not interesting, the special case $c=0$ of Theorem \ref{genc} is, since it gives a nice relation between the divisor function $d(n)$ and $\textup{lpt}(n)$, that is, the total number of appearances of largest parts in all partitions of $n$.
\begin{corollary}\label{c0}
We have
\begin{equation}\label{c0eqn}
\sum_{n=1}^{\infty}\frac{q^n}{1-q^n}+\sum_{n=1}^{\infty}\frac{q^{n(n+1)}}{(1-q^n)(q)_n^2}=\sum_{n=1}^{\infty}\frac{q^n}{(1-q^n)(q)_n}.
\end{equation}
Let $\nu(i)$ denote the number of appearances of $i$ in a partition of a positive integer with the understanding that $\nu(l(\pi))=L(\pi)$. Define $w(n)$ by
{\allowdisplaybreaks\begin{equation}\label{wn}
w(n):=\sum_{\pi\in\mathcal{P}^{*}(n)\atop \nu(i)\geq 2, 1\leq i\leq l(\pi)}\frac{L(\pi)(L(\pi)-1)}{2}\prod_{i=1}^{l(\pi)-1}(\nu(i)-1).
\end{equation}}
Then \eqref{c0eqn} implies
\begin{equation}\label{dwlpt}
d(n)+w(n)=\textup{lpt}(n).
\end{equation}
\end{corollary}
In his fascinating work on Vassiliev invariants, Zagier \cite[Theorem 2]{zagiertop} found a nice `sum of tails identity':
\begin{equation}\label{zagierid}
\sum_{n=0}^{\infty}\left((q)_n-(q)_{\infty}\right)=-\frac{1}{2}H(q)+(q)_{\infty}\left(\frac{1}{2}-E(q)\right),
\end{equation} 
where
\begin{align*}
H(q):=\sum_{n=1}^{\infty}n\left(\frac{12}{n}\right)q^{\frac{n^2-1}{24}}\hspace{2mm}\text{and}\hspace{2mm}E(q):=\sum_{n=1}^{\infty}\frac{q^n}{1-q^n}.
\end{align*}
Using this result as well as our Corollary \ref{c0}, we find an explicit representation for $w(n)$ in terms of the divisor and partition functions.
\begin{corollary}\label{wdpn}
Let $w(n)$ be defined in \eqref{wn}. Then
\begin{equation}\label{wnrep}
w(n)=-2d(n)-\frac{1}{2}\sum_{k=1}^{\left\lfloor\sqrt{24n+1}\right\rfloor}k\left(\frac{12}{k}\right)p\left(n-\frac{(k^2-1)}{24}\right),
\end{equation}
with the understanding that $p(x)=0$ if $x$ is not a positive integer.
\end{corollary}
As defined in \cite{agl13}, let $V$ denote the set of vector partitions, that is, $V=\mathcal{D}\times\mathcal{P}\times\mathcal{P}$, where $\mathcal{P}$ denotes the set of partitions and $\mathcal{D}$ denotes the set of partitions into distinct parts. Let $S$ denote the following set of vector partitions:
\begin{equation*}
S:=\{\vec{\pi}=(\pi_1, \pi_2, \pi_3)\in V: 1\leq s(\pi_1)<\infty\hspace{1mm}\text{and}\hspace{1mm}s(\pi_1)\leq\min(s(\pi_2), s(\pi_3))\}.
\end{equation*}
Let $\omega_1(\vec{\pi})=(-1)^{\#(\pi_1)-1}$ and define the involution map $\imath: S\to S$ by
\begin{equation*}
\imath(\vec{\pi})=\imath(\pi_1, \pi_2, \pi_3)=\imath(\pi_1, \pi_3, \pi_2).
\end{equation*}
Define an $S$-partition $\vec{\pi}=(\pi_1, \pi_2, \pi_3)$ to be a self-conjugate $S$-partition if it is a fixed point of $\imath$, that is, if and only if $\pi_2=\pi_3$. Let $N_{\textup{SC}}(n)$ denote the number of self-conjugate $S$-partitions counted according to the weight $\omega_1$, that is,
\begin{equation}\label{nscofn}
N_{\textup{SC}}(n)=\sum_{\vec{\pi}\in S, |\vec{\pi}|=n \atop \imath(\vec{\pi})=\vec{\pi}}\omega_1(\vec{\pi}).
\end{equation}
Andrews, Garvan and Liang \cite[Theorem 1.2]{agl13} showed that the generating function of $N_{\textup{SC}}(n)$ equals $\displaystyle\sum_{n=1}^{\infty}\frac{(-1)^{n-1}q^{n^2}}{(q;q^2)_n}$, which is a mock theta function studied by Andrews, Dyson and Hickerson \cite{adh}.

If we let $c=-1$ in Theorem \ref{genc}, we obtain a result involving the generating function of $N_{\textup{SC}}(n)$. 
\begin{corollary}\label{nsc}
Let $N_{\textup{SC}}(n)$ be defined in \eqref{nscofn}. Then
\allowdisplaybreaks{\begin{align}\label{nsceqn}
&(q)_{\infty}\sum_{n=1}^{\infty}N_{\textup{SC}}(n)q^n+\frac{1}{2}\frac{(q)_{\infty}}{(-q)_{\infty}}\sum_{n=1}^{\infty}\frac{q^{\frac{n(n+1)}{2}}}{(1-q^n)(q)_n}\left(\frac{(-q)_n}{(q)_n}-1\right)\nonumber\\
&=\frac{1}{4}-\frac{1}{4}\frac{(q)_{\infty}}{(-q)_{\infty}}+\frac{1}{2}\frac{(q)_{\infty}}{(-q)_{\infty}}\sum_{n=1}^{\infty}\frac{(-q)_n}{(q)_n}\frac{q^n}{1-q^n}.
\end{align}}
\end{corollary}
Using a result conjectured by Beck and proved by Chern \cite[Theorem 1.2]{chern}, we show in Section \ref{diffe} that the above identity reduces to a new representation for the generating function of $d(n)$.
\begin{corollary}\label{newdn}
We have
\begin{equation}\label{newdneqn}
\sum_{n=1}^{\infty}\frac{(-q)_n}{(q)_n}\frac{q^n}{1-q^n}-2\sum_{n=1}^{\infty}\frac{(-q)_{n-1}}{(q)_n^{2}}\frac{q^{n(n+3)/2}}{1-q^n}=\sum_{n=1}^{\infty}\frac{q^n}{1-q^n}.
\end{equation}
Hence
\begin{equation}\label{newdnrep}
\sum_{\pi\in\mathcal{P}_{o}(n) \atop l(\pi)\text{overlined}}(2L(\pi)-1)-\sum_{\pi\in\mathcal{P}^{*}(n)\atop L(\pi)\geq 2}L(\pi)(L(\pi)-1)\prod_{i=1}^{l(\pi)-1}(2\nu(i)-1)=d(n).
\end{equation}
\end{corollary}
Further special cases of Theorem \ref{genc} when $c=q$, and more generally, $c=q^m, m\geq 1$, are given in Section \ref{diffe}.

The last of the five $q$-series identities of Ramanujan \cite[p.~355]{ramanujanoriginalnotebook2}, \cite[p.~265, Entry 5]{bcbramforthnote} given by 
\begin{align}\label{fifth}
\sum_{n=1}^{\infty} \frac{(q)_{n-1} a^n}{(1-q^n) (a)_n }=\sum_{n=1}^{\infty} \frac{n a^n}{1-q^n},
\end{align}
where $(a)_n\neq 0, n\geq 1$, was proved by Berndt as a special case of \eqref{neglected}. In the same vein, our Theorem \ref{gen of Ramanujan's identity} gives the following generalization of \eqref{fifth}:
\begin{theorem}\label{g5}
For $|zq|<1$ and $|c|\leq 1$,
\begin{align}\label{generalization of Entry 5}
\sum_{n=1}^{\infty}   \frac{(q)_{n-1} z^nq^n }{ (1-c q^n) (zq)_n } &=z\sum_{n=1}^{\infty}\frac{(zq/c)_{n-1}}{(zq)_n}\frac{c^{n-1}q^n}{1-zq^{n}}.
\end{align}
Additionally if $|zq|<|c|$, then
\begin{align}\label{generalization of Entry 5o}
\sum_{n=1}^{\infty}   \frac{(q)_{n-1} z^nq^n }{ (1-c q^n) (zq)_n }= \frac{ (\frac{zq}{c})_\infty}{(zq)_\infty} \sum_{n=0}^{\infty} \frac{(c)_n \left( \frac{zq}{c} \right)^n }{(q)_n}    \sum_{m=1}^{\infty}   \frac{m z^mq^m}{1- c q^{m+n}  }.
\end{align}
\end{theorem}
This result gives a nice weighted partition identity given below.
\begin{theorem}\label{g5wpi}
If $z$ and $c$ are not a function of $q$, then 
\begin{align*}
\sum_{\pi \in \mathcal{P}(n) } z^{l(\pi)+ \#(\pi) - L(\pi)} \left(1- \tfrac{1}{z}\right)^{\nu_d(\pi)-1} \left(\tfrac{z^{L(\pi)} -c^{L(\pi)} }{z - c}\right) = \sum_{\pi \in \mathcal{P}(n) } c^{l(\pi)-1} z^{\#(\pi)} \left(1- \tfrac{1}{c}\right)^{\nu_d(\pi)-1} L(\pi) .
\end{align*}
\end{theorem}
From the above theorem, we obtain two beautiful weighted partition identities, the first of which is
\begin{corollary}\label{cl}
The following identity holds:
\begin{align}\label{cleqn}
\sum_{n=1}^{\infty} \frac{ q^n }{1- q^{2n}} = \sum_{n=1} (-1)^{n-1} \frac{(-q)_{n-1}}{ (q)_{n}} \frac{q^n}{1-q^n}.
\end{align}
Hence if $d_o(n)$ denotes the number of odd divisors of $n$,
\begin{align}
d_{o}(n) & = \sum_{\pi \in \mathcal{P}(n)} (-1)^{l(\pi)-1} 2^{\nu_d(\pi)-1} L(\pi) \label{clpti}
\end{align}
\end{corollary}
Though Corteel and Lovejoy \cite[p.~1631]{overp} have obtained \eqref{cleqn}, the identity \eqref{clpti} appears to be new. Garvan obtained another weighted partition representation for $d_{o}(n)$ in \cite[Corollary (1.3) (i)]{garvan1}, and a yet another one is given in Section \ref{weight}.

The second result that we obtain using \eqref{generalization of Entry 5} is
\begin{corollary}\label{cli}
The following identity holds:
\begin{align*}
\sum_{n=1}^{\infty} \frac{ n (-1)^n q^n }{1 - q^n } = \sum_{n=1}^{\infty} \frac{(-1)^n (q)_{n-1} }{(-q)_{n-1}}
\frac{q^n}{1- q^{2n}}.
\end{align*}
Hence
\begin{align*}
\sum_{d |n } (-1)^d d = \sum_{\substack{ \pi \in \mathcal{P}(n) \\ L(\pi)\, \mathrm{is\, odd} } } (-1)^{\textup{rank}(\pi)-1} 2^{\nu_{d}(\pi) - 1}.
\end{align*} 
\end{corollary}

This paper is organized as follows. In Section \ref{prelim} we collect some of the fundamental results in $q$-series that will be used several times in the paper. Section \ref{entry3gene} is devoted to the proofs of Theorems \ref{entry3gen}, \ref{gGaravan} and \ref{g5}. In Section \ref{weight}, we give proofs of the weighted partition identities mentioned here in Section \ref{mare}. Many other weighted partition identities are also derived here. Section \ref{app} is reserved for proving Garvan's identity \eqref{garvan's 2nd thm}, obtaining number-theoretic information from its proof, and for giving another proof of Andrews, Garvan and Liang's \eqref{ffwcn} using the Bhargava-Adiga summation. In Section \ref{diffe}, we prove Theorem \ref{genc} and its corollaries. Finally the concluding remarks and possibilities for future work are addressed in Section \ref{cr}.

\section{Preliminary results}\label{prelim}
For $|z|<1$, the $q$-binomial theorem is given by \cite[p.~17, Equation (2.2.1)]{gea1998}
\begin{equation}\label{q-binomial thm}
\sum_{n=0}^{\infty}\frac{(a)_{n}z^n}{(q)_n}=\frac{(az)_{\infty}}{(z)_{\infty}}.
\end{equation}
Let the Gaussian polynomial $\left[\begin{matrix} n\\m\end{matrix}\right]$ be defined by \cite[p.~35]{gea1998}
\begin{equation*}
\left[\begin{matrix} n\\m\end{matrix}\right]=\left[\begin{matrix} n\\m\end{matrix}\right]_{q}:=
\begin{cases}
(q;q)_{n}(q;q)_m^{-1}(q;q)_{n-m}^{-1},\hspace{2mm}\text{if}\hspace{1mm}0\leq m\leq n,\\
0,\hspace{2mm}\text{otherwise}.
\end{cases}
\end{equation*}
From \cite[p.~36, Equation (3.3.7)]{gea1998}, we have
\begin{align}\label{binomial}
\frac{1}{(z)_{N}}=\sum_{j=0}^{\infty}\left[\begin{matrix} N+j-1\\ j\end{matrix} \right]z^j.
\end{align}
The partial fraction decomposition of $F(a,b;t)$ is given by \cite[p. 18, Equation (16.3)]{fine}
\begin{equation}\label{fine16.3}
F(a,b;t)=\frac{(aq)_{\infty}}{(bq)_{\infty}}\sum_{n=0}^{\infty}\frac{(b/a)_n}{(q)_n}\frac{(aq)^n}{1-tq^n}.	
\end{equation}
Heine's transformation \cite[p.~19, Corollary 2.3]{gea1998} gives
\begin{align}\label{heine}
{}_{2}\phi_{1}\left( \begin{matrix} a, & b \\
& c \end{matrix} \, ; q, z  \right) = \frac{(b, a z; q)_{\infty }}{(c, z; q)_{\infty}} {}_{2}\phi_{1}\left( \begin{matrix} \frac{c}{b} , & z \\
& a z \end{matrix} \, ; q, b  \right).
\end{align}

\section{Generalizations of Ramanujan's three $q$-series identities}\label{entry3gene}
Theorem \ref{gen of Ramanujan's identity} is proved in this section. We then give some of its applications.

\begin{proof}[Theorem \textup{\ref{gen of Ramanujan's identity}}][]
Let
\begin{align*}
G(a, b; c):=\sum_{n=1}^{\infty}  \frac{ \left( \frac{b}{a} \right)_n a^n }{ (1- c q^n) (b)_n }.
\end{align*}
Note that
{\allowdisplaybreaks
\begin{align}\label{first difference}
G(a, b; c) -G(aq, bq; c) & = \sum_{n=1}^{\infty}  \left(\frac{ \left( \frac{b}{a} \right)_n a^n }{ (1- c q^n) (b)_n } -  \frac{ \left( \frac{b}{a} \right)_n (a q)^n }{ (1- c q^n) (bq )_n }\right) \nonumber \\
& = \sum_{n=1}^{\infty} \frac{ \left( \frac{b}{a} \right)_n a^n }{ (1- c q^n) (b)_{n+1} } ( 1- b q^n - q^n (1-b) ) \nonumber \\
& =  \sum_{n=1}^{\infty} \frac{ \left( \frac{b}{a} \right)_n a^n }{ (1- c q^n) (b)_{n+1} } (1-q^n ) \nonumber \\
& =  \sum_{n=1}^{\infty} \frac{ \left( \frac{b}{a} \right)_n a^n }{  (b)_{n+1} } \left( \frac{ (1- c q^n) - (1-c) q^n )}{ 1- c q^n } \right) \nonumber \\
& = \sum_{n=1}^{\infty} \frac{ \left( \frac{b}{a} \right)_n a^n }{  (b)_{n+1} } - (1-c) \sum_{n=1}^{\infty} \frac{ \left( \frac{b}{a} \right)_n (aq)^n }{   (1- c q^n) (b)_{n+1} } \nonumber \\
& =  \sum_{n=1}^{\infty} \frac{ \left( \frac{b}{a} \right)_n a^n }{  (b)_{n+1} } - \frac{(1-c)}{(1-b)} G(aq, bq; c).  
\end{align} }
We now use the result in Berndt \cite[p. 264, Equation (3.4)]{bcbramforthnote}, that is, 
\begin{align}\label{constant}
 \sum_{n=1}^{\infty} \frac{ \left( \frac{b}{a} \right)_n a^n }{  (b)_{n+1} } = \frac{a}{1-a} - \frac{b}{1-b},
\end{align}
which is obtained by specializing Heine's transformation and then employing the $q$-binomial theorem. Thus substituting \eqref{constant} in \eqref{first difference},
\begin{align}\label{first term}
G(a,b;c) -  \left( \frac{c-b}{1-b} \right)   G(aq, bq; c) =  \frac{a}{1-a} - \frac{b}{1-b}. 
\end{align}
We now create a telescoping sum. To that end, replace $a$ and $b$ by $aq$ and $bq $ respectively so that  
\begin{align*}
G(aq,bq;c) -  \left( \frac{c-bq}{1-bq} \right)   G(aq^2, bq^2; c) =  \frac{aq}{1-aq} - \frac{bq}{1-bq}. 
\end{align*}
Multiply both sides of the above equation by $(c-b)/(1-b)$ to get
\begin{align}\label{second term}
\left( \frac{c-b}{1-b} \right) G(aq, bq; c) - \left( \frac{c-b}{1-b} \right) \left( \frac{c-bq}{1-bq} \right)  G(aq^2, bq^2; c)= \left( \frac{c-b}{1-b} \right) \left( \frac{aq}{1-aq} - \frac{bq}{1-bq} \right).
\end{align}
Add the corresponding sides of \eqref{first term} and \eqref{second term} to obtain
\begin{align}\label{adding first and second}
G(a,b; c) - \frac{(c-b)(c-bq)}{(1-b)(1-bq)}  G(aq^2, bq^2; c) & = \left(\frac{a}{1-a} - \frac{b}{1-b}\right) + \frac{(c-b)}{(1-b)} \left( \frac{aq}{1-aq} - \frac{bq}{1-bq} \right).
\end{align}
Repeat this process, that is, first replace $a$ by $aq$ and $b$ by $b q$ in \eqref{second term}, multiply both sides by $(c-b)/(1-b)$, and then add the corresponding sides of the resulting identity and \eqref{adding first and second}. At $n^{\textup{th}}$ step, this gives
\begin{align*}
G(a, b; c) - \frac{c^{n+1} \left(\frac{b}{c} \right)_{n+1}}{(b)_{n+1} } G(a q^{n+1}, b q^{n+1} ; c) = \sum_{k=0}^{n} \frac{c^k \left( \frac{b}{c} \right)_k}{ (b)_k } \left(  \frac{a q^k}{1- a q^k } - \frac{b q^k }{1- b q^k } \right). 
\end{align*}
Now let $n \rightarrow \infty$. Since $|c|\leq 1$, the second expression on the left-hand side goes to zero. Therefore,
\begin{align*}
G(a, b; c) & = \sum_{k=0}^{\infty} \frac{c^k \left( \frac{b}{c} \right)_k}{ (b)_k } \left(  \frac{a q^k}{1- a q^k } - \frac{b q^k }{1- b q^k } \right) \nonumber \\
           & =  \sum_{k=0}^{\infty}\frac{c^k \left( \frac{b}{c} \right)_k}{ (b)_k } \sum_{m=1}^{\infty} (a^m - b^m) q^{km} \nonumber \\
           & = \sum_{m=1}^{\infty} (a^m - b^m ) \sum_{k=0}^{\infty} \frac{ \left( \frac{b}{c} \right)_k c^k q^{m k} }{ (b)_k }. 
\end{align*}
Invoking \eqref{fine16.3} in the above equation, we observe that
\begin{align*}
G(a, b; c) & = \frac{ \left( \frac{ b }{c } \right)_{\infty }}{(b)_{\infty} } \sum_{m=1}^{\infty} (a^m - b^m ) \sum_{n=0}^{\infty} \frac{(c)_n}{(q)_n} \frac{ \left( \frac{b}{c} \right)^n }{1- c q^{m+n}} \nonumber \\
           & =   \frac{ \left( \frac{ b }{c } \right)_{\infty }}{(b)_{\infty} }   \sum_{n=0}^{\infty} \frac{(c)_n \left( \frac{b}{c} \right)^n }{(q)_n}    \sum_{m=1}^{\infty} \frac{a^m - b^m }{1- c q^{m+n}  }.
\end{align*}
This completes the proof of Theorem \ref{gen of Ramanujan's identity}. 
\end{proof}

\begin{proof}[Theorem \textup{\ref{gGaravan}}][]
Let $a\to 0$ in \eqref{entry3gen}, use the elementary identity $\lim_{a\to 0}(b/a)_na^n=(-b)^nq^{n(n-1)/2}$, and then replace $b$ by $zq$ to arrive at \eqref{gen of Garavan}.
\end{proof}
We now prove the generalization of \eqref{fifth} given in Theorem \ref{g5}.
\begin{proof}[Theorem \textup{\ref{g5}}][]
To obtain \eqref{generalization of Entry 5}, divide both sides of \eqref{entry3gen} by $1-b/a$, let $b\to a$, then replace $a$ by $zq$ and simplify. Now, additionally if $|zq|<|c|$, perform the same operations on \eqref{entry3gen00} to obtain \eqref{generalization of Entry 5o}.
\end{proof}

\section{Weighted partition identities}\label{weight}
\subsection{Weighted partition identities resulting from Theorem \ref{gGaravan} and its corollaries}\label{wpigG}
We begin with a lemma which will be used several times in the sequel.
\begin{lemma}\label{d(n) in limit}
\begin{align*}
\sum_{\pi \in \mathcal{P}(n)} \lim_{z \rightarrow 1}  \left( 1 - \frac{1}{z}  \right)^{\nu_{d}(\pi) -1 } = d(n).
\end{align*} 
\end{lemma}
\begin{proof}

Note that
\begin{align*}
\lim_{z \rightarrow 1}  \left( 1 - \frac{1}{z}  \right)^{\nu_{d}(\pi) -1 } = \begin{cases}
1 & \quad \mathrm{if} \quad \nu_{d}(\pi) =1,\\
0 & \quad \mathrm{if} \quad \nu_{d}(\pi) >1.
\end{cases}
\end{align*}
Therefore we need to consider only those partitions $\pi$ of $n$ for which $\nu_{d}(\pi) =1$. If $m$ is a divisor of $n$ then we have $ n = m\ell$, for some positive integer $\ell$, so we can write $n$ as 
\begin{equation*}
n=\underbrace{m+m+\cdots+m}_{\ell\hspace{1mm}\text{times}}.
\end{equation*}
Therefore corresponding to every divisor of $n$ we can obtain a partition whose number of parts without multiplicity is $1$ and conversely, any partition $\pi$ of $n$ with $\nu_d(\pi)=1$ corresponds to a  divisor of $n$. This proves the lemma.
\end{proof}
\begin{proof}[Theorem \textup{\ref{gwpi}}][]
To prove that the left-hand side of \eqref{gen of Garavan} is the generating function of that of \eqref{gwpieqn} we first write the former as
\begin{align}\label{first step}
\sum_{n=1}^{\infty} \frac{(-1)^{n-1} z^n q^{\frac{n(n+1)}{2} } }{(1-c q^n) (z q)_n  }   =  \sum_{n=1}^\infty \frac{(-1)^{n-1} z^n q^{1+2+\cdots+n}}{(1-zq)(1-zq^2)\cdots(1-zq^n)(1-c q^n)}
\end{align}
and use
 \begin{align*}
 \frac{1}{1- z q^i } = \sum_{k_i = 0}^\infty z^{k_i} q^{i\, k_i},
 \end{align*}
for $1\leq i \leq n-1$, as well as \cite[p.~214]{agl13} 
\begin{align*}
\frac{q^n}{(1-z q^n) (1- c q^n)} =  \sum_{k_n = 1}^{\infty} \frac{\left(\frac{z}{c} \right)^{k_n} - 1}{ \left(\frac{z}{c} \right)-1} c^{k_n - 1} q^{n\, k_n}
\end{align*}
in the summand of the series on the right-hand side of \eqref{first step}. Thus
\begin{align*}
&\sum_{n=1}^{\infty} \frac{(-1)^{n-1} z^n q^{\frac{n(n+1)}{2} } }{(1-c q^n) (z q)_n  }\nonumber\\
&=\sum_{n=1}^\infty  \sum_{\substack{k_i=0, 1\leq i\leq n-1 \\ k_n =1}}^{\infty} (-1)^{n-1} z^{n +k_1+ k_2 + \cdots + k_{n-1}} c^{k_n -1}  \frac{\left(\frac{z}{c} \right)^{k_n} - 1}{ \left(\frac{z}{c} \right)-1} q^{ \frac{n(n-1)}{2} + k_1 +2 k_2 + \cdots + n k_n }.
\end{align*}
Now write the exponent of $q$ in the form $a_1+a_2\cdots+a_n$ where
\begin{align*}
a_1& := (k_1 +1) + (k_2+1) + \cdots+ (k_{n-1}+1) + k_n \\
a_2 & := (k_2 +1) + (k_3 + 1) + \cdots + (k_{n-1}+1) + k_n \\
& \quad \cdots \cdots\cdots \\
a_{n-1} & : = k_{n-1}+1 + k_n \\
a_n & := k_n.
\end{align*}
Thus $\{ a_1, a_2, \cdots, a_n\} $ is a distinct partition of the exponent of $q$, where the largest part is $a_1$, smallest part is $a_n$ and $n$ is the number of parts. This proves that
\begin{align}\label{distinct part}
\sum_{n=1}^{\infty} \frac{(-1)^{n-1} z^n q^{\frac{n(n+1)}{2} } }{(1-c q^n) (z q)_n  }=\sum_{n=1}^{\infty} \left( \sum_{\pi \in \mathcal{D}(n) } (-1)^{\#(\pi)-1} z^{l(\pi) + 1 -s(\pi)} c^{s(\pi)-1}\frac{\left(\frac{z}{c} \right)^{s(\pi)} - 1}{ \left(\frac{z}{c} \right)-1}  \right) q^n.
\end{align}
Next we write the right-hand side of \eqref{gen of Garavan} in the form
\begin{align*}
\frac{z}{c}\sum_{n=1}^{\infty}\frac{(zq/c)_{n-1}}{(zq)_n}(cq)^n = \frac{z}{c}
\sum_{n=1}^{\infty} \frac{\left(1- \frac{z}{c} q \right)\left(1- \frac{z}{c} q^2 \right) \cdots \left(1- \frac{z}{c} q^{n-1} \right)}{(1- zq) (1- z q^2) \cdots (1- z q^{n-1} ) } \frac{c^n q^n }{(1- z q^n )}.
\end{align*}
For $ 1 \leq i \leq n-1$, we have 
\begin{align*}
\frac{\left(1- \frac{z}{c} q^i \right)}{\left(1- z q^i \right)} = 1 + \sum_{k_i =1}^\infty \left(1- \frac{1}{c} \right) z^{k_i} q^{i\, k_i} = \sum_{k_i=0}^\infty a_{k_i} z^{k_i} q^{i\, k_i},
\end{align*}
where $$ a_{k_i} = \begin{cases} 1 \quad & \quad \mathrm{if}\quad k_i=0, \\
1- \frac{1}{c} \quad & \quad \mathrm{if} \quad k_i \geq 1.
\end{cases}  $$
Furthermore, write
\begin{align*}
\frac{ z q^n}{ 1- z q^n } =  \sum_{k_n=1}^{\infty} z^{k_n} q^{n\, k_n}. 
\end{align*}
Using these series expansions in the above identity, we get
\begin{align*}
\frac{z}{c}\sum_{n=1}^{\infty}\frac{(zq/c)_{n-1}}{(zq)_n}(cq)^n=\sum_{n=1}^{\infty} c^{n-1} \prod_{i=1}^{n-1} \left( \sum_{k_i=0}^\infty a_{k_i} z^{k_i} q^{i\, k_i} \right)   \sum_{k_n=1}^{\infty} z^{k_n} q^{n\, k_n}. 
\end{align*}
Here the typical exponent of $q$ is $k_1 + 2 k_2 + \cdots + n k_n$, where $k_i \geq 0 $ for $1 \leq i \leq n-1$ and $k_n \geq 1$. This corresponds to an ordinary partition of the exponent where, every part $i$, $1 \leq i \leq n$, appears $k_i$ times, the largest part is $n$ and the number of parts is $k_1 + k_2 + \cdots + k_n$. Thus
\begin{align}\label{ordinary partition}
\frac{z}{c}\sum_{n=1}^{\infty}\frac{(zq/c)_{n-1}}{(zq)_n}(cq)^n=\sum_{n=1}^\infty \left( \sum_{\pi \in \mathcal{P}(n)} c^{l(\pi) -1 } \left( 1- \frac{1}{c}  \right)^{\nu_{d} (\pi) -1 } z^{\#(\pi)} \right) q^n. 
\end{align}
From \eqref{gen of Garavan}, \eqref{distinct part} and \eqref{ordinary partition}, we obtain \eqref{gwpieqn}. 
\end{proof}

When $c=z$, Theorems \ref{gGaravan} and \ref{gwpi} give the following corollary.
\begin{corollary}\label{gGaravancez}
For $|z|\leq 1$,
\begin{align}\label{gen of Garavancez}
\sum_{n=1}^{\infty} \frac{(-1)^{n-1} z^n q^{\frac{n(n+1)}{2} } }{(1-z q^n) (z q)_n  } = \sum_{n=1}^{\infty}\frac{(q)_{n-1}}{(zq)_n}(zq)^n.
\end{align}
Hence if $z$ is not a function of $q$,
\begin{align*}
 \sum_{\pi \in \mathcal{D}(n) } (-1)^{\#(\pi)-1} z^{l(\pi)} s(\pi)=\sum_{\pi \in \mathcal{P}(n) } z^{\#(\pi)+l(\pi) -1} \left(1- \frac{1}{z}  \right)^{\nu_d(\pi) -1 }.
\end{align*}
\end{corollary}
We note that \eqref{gen of Garavancez} can also be obtained by letting $c=0$ in the first equality of Theorem \ref{g5}.

The above corollary, in turn, specializes to Kluyver's identity and the result of Fokkink, Fokkink and Wang:
\begin{corollary}\label{gGaravancez1}
The identities \eqref{Uchimura} and \eqref{ffwidty} hold.
\end{corollary}
\begin{proof}
Let $z=1$ in Corollary \ref{gGaravancez}, and employ Lemma \ref{d(n) in limit} to get the second equality.
\end{proof}
Corollary \ref{gGaravancez} also gives a new weighted partition identity.
\begin{corollary}\label{gGaravancez-1}
We have
\begin{align*}
\sum_{n=1}^{\infty} \frac{q^{\frac{n(n+1)}{2} } }{(1+ q^n) (-q)_n  } = \sum_{n=1}^{\infty}\frac{(-1)^{n-1}(q)_{n-1}}{(-q)_n}q^n.
\end{align*}
Hence,
\begin{align}\label{rankspi}
 \sum_{\pi \in \mathcal{D}(n)} (-1)^{\mathrm{rank}(\pi)} \, s(\pi) = \sum_{\pi\in\mathcal{P}(n) }  (-1)^{\mathrm{rank}(\pi)} \, 2^{{\nu_{d}(\pi)}-1}.
\end{align}
\end{corollary}
\begin{proof}
Let $z=-1$ in Corollary \ref{gGaravancez}, and to obtain \eqref{rankspi}, use $(-1)^{\#(\pi)+l(\pi)}=(-1)^{l(\pi)-\#(\pi)}=(-1)^{\mathrm{rank}(\pi)}$.
\end{proof}
\begin{example}
Let $n=5$. In the first table below, we list the partitions of $n$ into distinct parts, and in the other, all partitions.

\begin{tabular}{|c|c|c|c|}
\hline 
$\pi \in \mathcal{D}_5$ & $\mathrm{rank(\pi)}$ & $(-1)^{\mathrm{rank}(\pi)} s(\pi)$  \\
\hline
$5$ & $4$ & $5$ \\
\hline
$4+1$ & $2$ & $1$ \\
\hline
$3+2$ & $1$ & $-2$ \\
\hline
\end{tabular}
\begin{tabular}{|c|c|c|}
\hline 
$\pi \in \mathcal{P}(5)$ & $\mathrm{rank(\pi)}$ & $(-1)^{\mathrm{rank}(\pi)}2^{\nu_d(\pi)-1} $   \\
\hline
$5$ & $4$ & $1$ \\
\hline
$4+1$ & $2$ & $2$ \\
\hline
$3+2$ & $1$ & $-2$ \\
\hline
$3+1+1$ & $0$ & $2$ \\
\hline 
$2+2+1$ & $-1$ &  $-2$ \\
\hline
$2+1+1+1$ & $-2$ & $2$ \\
\hline
$1+1+1+1+1+1$ & $-4$ & $1$ \\
\hline
\end{tabular}
It can be checked that the sum of the last columns in each of the tables is equal to $4$, which verifies \eqref{rankspi} for $n=5$. 
\end{example}
Similarly when $c=-z$, Theorems \ref{gGaravan} and \ref{gwpi} result in the following.
\begin{corollary}\label{gGaravance-z}
For $|z|\leq 1$,
\begin{align}\label{gen of Garavance-z}
\sum_{n=1}^{\infty} \frac{(-1)^{n-1} z^n q^{\frac{n(n+1)}{2} } }{(1+zq^n) (z q)_n  } = \sum_{n=1}^{\infty}\frac{(-1)^{n-1}(-q)_{n-1}}{(zq)_n}(zq)^n.
\end{align}
Hence if $z$ is not a function of $q$,
\begin{align}\label{gen of Garavance-zpi}
\sum_{\pi \in \mathcal{D}(n) \atop s(\pi)\hspace{0.5mm}\text{odd}} (-1)^{\#(\pi)-1}z^{l(\pi)}=\sum_{\pi\in\mathcal{P}(n)}(-1)^{l(\pi)-1}z^{\#(\pi)+l(\pi)-1}\left(1+\frac{1}{z}\right)^{\nu_d(\pi)-1}.
\end{align}
\end{corollary}
\begin{proof}
Equation \eqref{gen of Garavance-z} is immediate after we let $c=-z$ in Theorem \ref{gGaravan}. To obtain \eqref{gen of Garavance-zpi}, let $c=-z$ in Theorem \ref{gwpi} and note that the resulting left-hand side is non-zero only when $s(\pi)$ is odd.
\end{proof}
Now if we let $z=1$ in the above corollary, we obtain
\begin{corollary}
\begin{align*}
\sum_{n=1}^{\infty} \frac{(-1)^{n-1} q^{\frac{n(n+1)}{2} } }{(1+q^n) (q)_n  } = \sum_{n=1}^{\infty}\frac{(-1)^{n-1}(-q)_{n-1}}{(q)_n}q^n.
\end{align*}
Hence
\begin{align}\label{alla}
\sum_{\pi \in \mathcal{D}(n) \atop s(\pi)\hspace{0.5mm}\textup{odd}} (-1)^{\#(\pi)-1}=\sum_{\pi\in\mathcal{P}(n)}(-1)^{l(\pi)-1}2^{\nu_d(\pi)-1}.
\end{align}
\end{corollary}
We note that in Andrews, Garvan and Liang's notation in \eqref{ffwcn}, the left-hand side of the above weighted partition identity is nothing but $\textup{FFW}(-1,n)$. As mentioned in \cite[Remark 3.7]{agl13}, Alladi \cite[Theorem 2]{alladipartialram} has obtained a simpler representation for the left-hand side of \eqref{alla}, namely, it is equal to $(-1)^{\sqrt{n}-1}$ if $n$ is a square and zero otherwise.

Lastly if we let $z\to-1$ in Corollary \ref{gGaravance-z}, we obtain
\begin{corollary}
\begin{align}\label{gen of Garavance-z-1}
\sum_{n=1}^{\infty} \frac{q^{\frac{n(n+1)}{2} } }{(1-q^n) (-q)_n  } = \sum_{n=1}^{\infty}\frac{q^n}{1+q^n}.
\end{align}
Hence if $d_{e}(n)$ and $d_{o}(n)$ denote the number of even divisors and the number of odd divisors of $n$ respectively,
\begin{align}\label{gen of Garavance-z-1pi}
\sum_{ \pi \in \mathcal{D}(n) \atop s(\pi)\, \mathrm{odd}  } (-1)^{\mathrm{rank}(\pi)-1}=d_{e}(n) - d_{o}(n).
\end{align}
\end{corollary}
\begin{proof}
Letting $z\to -1$ in \eqref{gen of Garavance-z} immediately results in \eqref{gen of Garavance-z-1}. To obtain \eqref{gen of Garavance-z-1pi}, let $z\to -1$ in \eqref{gen of Garavance-zpi} and observe that analogous to the result in Lemma \ref{d(n) in limit}, the following limit holds:
\begin{equation*}
\lim_{z\to -1}\sum_{\pi\in\mathcal{P}(n)}z^{\#(\pi)}\left(1+\frac{1}{z}\right)^{\nu_d(\pi)-1}=d_{e}(n) - d_{o}(n).
\end{equation*}
\end{proof}
\begin{example}
Suppose $n=10$. Below we list the distinct partitions of $10$ such that their smallest parts are odd.
\begin{center}
\begin{tabular}{|c|c|c|}
\hline 
$\pi$ & $\mathrm{rank(\pi)}$ & $(-1)^{\mathrm{rank}(\pi)}$ \\
\hline
$9+1$ & $7$ & $-1$ \\
\hline
$7+3$ & $5$ & $-1$ \\
\hline
$7+2+1$ & $4$ & $+1$ \\
\hline
$6+3+1$ & $3$ & $-1$ \\
\hline
$5+4+1$ & $2$ & $+1$ \\
\hline
$4 +3 +2+1$ & $0$ & $+1$ \\
\hline
\end{tabular}

\end{center}
It can be easily checked that $d_{o}(10)= d_{e}(10)=2$ and that the sum of the last column equals zero. This verifies \eqref{gen of Garavance-z-1pi} for $n=10$. 
\end{example}
\begin{remark}
In all of the weighted partition representations of expressions in the above theorems involving partitions into distinct parts, one can replace $\mathcal{D}(n)$ by $\mathcal{P}^{*}(n)$ and interchange $l(\pi)$ with $\#(\pi)$, and $s(\pi)$ with $L(\pi)$. This simply follows from conjugation of partitions. See, for example, Ando's representation \cite[p.~2]{andoelec} of the left-hand side \eqref{ffwidty}.
\end{remark}
\subsection{Weighted partition identities resulting from Theorem \ref{g5} and its corollaries}\label{wpig5}
The details of its derivation of Theorem \ref{g5wpi} from Theorem \ref{g5} are similar to that of Theorem \ref{gwpi} from Theorem \ref{gGaravan}, hence a proof is omitted. 

Note that letting $c=z$ in \eqref{generalization of Entry 5} does not produce anything interesting, however, $c=-z$ gives the following result.
\begin{corollary}\label{cequal-z}
For $|z|\leq 1$,
\begin{align*}
\sum_{n=1}^{\infty}   \frac{(q)_{n-1} z^nq^n }{ (1 + z q^n) (zq)_n } &=z\sum_{n=1}^{\infty}\frac{(-q)_{n-1}}{(zq)_n}\frac{(-z)^{n-1}q^n}{1-zq^{n}}.
\end{align*}
Hence 
\begin{align*}
\sum_{\pi \in \mathcal{P}(n) \atop L(\pi)\hspace{1mm}\text{odd}} z^{l(\pi)+ \#(\pi) - 1} \left(1- \frac{1}{z}\right)^{\nu_d(\pi)-1} = \sum_{\pi \in \mathcal{P}(n) } (-z)^{l(\pi)-1} z^{\#(\pi)} \left(1+ \tfrac{1}{z}\right)^{\nu_d(\pi)-1} L(\pi). 
\end{align*}
\end{corollary}
Let $z=1$ in Theorems \ref{g5} and \ref{g5wpi}. This gives
\begin{corollary}
For $|c|\leq 1$, we have
\begin{align*}
\sum_{n=1}^{\infty}   \frac{ q^n }{ (1-c q^n) (1-q^n) } = \sum_{n=1}^{\infty}\frac{(q/c)_{n-1}}{(q)_n}\frac{c^{n-1}q^n}{1-q^{n}}.
\end{align*}
Hence if $c$ is not a function of $q$,
\begin{align*}
\sum_{d|n} \frac{c^d -1 }{c-1} = \sum_{\pi \in \mathcal{P}(n)} c^{l(\pi) -1}  \left(1- \frac{1}{c}\right)^{\nu_d(\pi)-1} L(\pi). 
\end{align*}
\end{corollary}  

\begin{proof}[Corollary \textup{\ref{cl}}][]
Let $c=-1$ in the above corollary.
\end{proof}
\begin{remark}
One more representation of $d_{o}(n)$ that one can obtain is
\begin{equation*}
d_{o}(n)=\frac{1}{2} \sum_{ \pi \in \mathcal{P}_{o}(n)} (-1)^{l(\pi) -1 } L(\pi).
\end{equation*}
To see this, write the right-hand side of \eqref{cleqn} as
\begin{align}\label{ram}
\sum_{n=1}^\infty (-1)^{n-1} \frac{(-q)_{n-1}}{ (q)_{n-1}} \frac{q^n}{(1-q^n)^2}=\frac{1}{2}\sum_{n=1}^{\infty}  \frac{2 (-1)^{n-1} (-q)_{n-1}}{(q)_{n-1}} \sum_{k=1}^{\infty} k q^{n k}.
\end{align}
It is well-known that the generating function for overpartitions is 
\begin{align*}
\sum_{n=0}^{\infty} \bar{p}(n) q^n = \frac{(-q)_\infty}{(q)_\infty} = 1+ 2 \sum_{n=1}^{\infty} \frac{(-q)_{n-1}}{(q)_{n-1} } \frac{q^n}{1- q^n },
\end{align*}
where $\bar{p}(n)$ denotes the number of over partitions of $n$. Thus, the presence of $(1-q^n)^2$ in the denominator of \eqref{ram} suggests that the series is the generating function of the total number of appearances of the largest part $L(\pi)$ of a partition $\pi$ weighted by $(-1)^{l(\pi) -1}$.
\end{remark}
Now let $z=-1$ in Theorems \ref{g5} and \ref{g5wpi} to obtain
\begin{corollary}
For $|c|\leq 1$,
\begin{align*}
\sum_{n=1}^{\infty}   \frac{(q)_{n-1}(-1)^n q^n }{ (1-c q^n) (-q)_n } &=-\sum_{n=1}^{\infty}\frac{(-q/c)_{n-1}}{(-q)_{n-1}}\frac{c^{n-1}q^n}{(1+q^{n})^2}.
\end{align*}
 Hence if $c$ is not a function of $q$,
\begin{align*}
\sum_{\pi \in \mathcal{P}(n) } (-1)^{\mathrm{rank}(\pi) - L(\pi)} 2^{\nu_d(\pi)-1} \tfrac{(-1)^{L(\pi)} -c^{L(\pi)} }{-1 - c} = \sum_{\pi \in \mathcal{P}(n) } c^{l(\pi)-1} (-1)^{\#(\pi)} \left(1- \tfrac{1}{c}\right)^{\nu_d(\pi)-1} L(\pi). 
\end{align*}
\end{corollary}
\begin{proof}[Corollary \textup{\ref{cli}}][]
Let $c=1$ in the above corollary and observe that
\begin{equation*}
-\sum_{n=1}^{\infty}\frac{q^n}{(1+q^{n})^2} = \sum_{n=1}^{\infty} \frac{ n (-1)^n q^n }{1 - q^n }.
\end{equation*}
\end{proof}

\section{Some more applications of Theorem \ref{gGaravan}}\label{app}
\subsection{Proof of Garvan's identity \eqref{garvan's 2nd thm}}\label{g2t}
Replace $q$ by $q^2$, $z$ by $z/q$, and then $c$ by $z$ in \eqref{gen of Garavan}. This gives
\begin{align}\label{new identity}
\sum_{n=1}^{\infty}  \frac{(-1)^{n-1} z^n q^{n^2} }{(zq;q^2)_{n} (1-z q^{2n}) } = \sum_{m=1}^{\infty}  z^{m} q^{2 m - 1 }  \frac{(q;q^2)_{m-1}}{(zq;q^2)_{m}}.
\end{align}
An application of \eqref{binomial} to the reciprocal of $(zq;q^2)_m$ on the right-hand side of the above equation and then of \eqref{heine} in the third step below gives
{\allowdisplaybreaks\begin{align}\label{use fine result} 
\sum_{m=1}^{\infty}  z^{m} q^{2 m - 1 }  \frac{(q;q^2)_{m-1}}{(zq;q^2)_{m}}&= \sum_{m=1}^{\infty} z^{m} q^{2 m - 1 } (q;q^2)_{m-1} \sum_{j=0}^{\infty} \frac{(q^2; q^2)_{m+j-1} }{(q^2;q^2)_j (q^2; q^2)_{m-1} } (z q)^j \nonumber \\
&=z q \sum_{j=0}^{\infty} (z q)^j {}_{2}\phi_{1} \left(  \begin{matrix} q, & q^{2j+2} \\
& 0 \end{matrix}\, ;q^2,  z q^2 \right)\nonumber\\
& =z q \sum_{j=0}^{\infty} (z q)^j   \frac{(q^{2j+2}; q^2)_{\infty} (z q^3; q^2)_{\infty}}{(z q^2; q^2)_{\infty}} \sum_{k=0}^{\infty} \frac{(zq^2; q^2)_k (q^{2j+2} )^k }{ (z q^3; q^2)_k (q^2; q^2)_k } \nonumber\\
&= z q \frac{(q^{2}; q^2)_{\infty} (z q^3; q^2)_{\infty}}{(z q^2; q^2)_{\infty}} \sum_{k=0}^{\infty} \frac{(zq^2; q^2)_k q^{2k} }{ (z q^3; q^2)_k (q^2; q^2)_k } \sum_{j=0}^{\infty} \frac{(z q^{2k+1})^j}{(q^2; q^2)_j }\nonumber\\
&= z q \frac{(q^{2}; q^2)_{\infty} (z q^3; q^2)_{\infty}}{(z q^2; q^2)_{\infty}} \sum_{k=0}^{\infty} \frac{(zq^2; q^2)_k q^{2k} }{ (z q^3; q^2)_k (q^2; q^2)_k } \frac{1}{(z q^{2k+1}; q^2)_{\infty}}\nonumber\\
&=z q \frac{(q^{2}; q^2)_{\infty} }{(z q^2; q^2)_{\infty} } \sum_{k=0}^{\infty}  \frac{ (zq^2; q^2)_k  }{ (q^2; q^2)_k } \frac{q^{2k}}{1-z q^{2k+1}},
\end{align}}
where in the penultimate step, we applied \eqref{q-binomial thm}.
Now replace $q $ by $q^2$, $t$ by $zq$, $a$ by $1$ and $b$ by $z q^2$ in  \eqref{fine16.3} so that the right-hand side of \eqref{use fine result} can be simplified to
\begin{align}\label{third expression}
z q \frac{(q^{2}; q^2)_{\infty} }{(z q^2; q^2)_{\infty} } \sum_{k=0}^{\infty}  \frac{ (zq^2; q^2)_k  }{ (q^2; q^2)_k } \frac{q^{2k}}{1-z q^{2k+1}}&=\frac{zq}{1-zq^2}\sum_{k=0}^{\infty} \frac{(q^2; q^2)_k }{(z q^4; q^2)_k} (z q)^k\nonumber\\
& = \sum_{k=1}^{\infty} \frac{(q^2; q^2)_{k-1} }{(z q^2; q^2)_{k}} (z q)^{k}.
\end{align}
Thus, from \eqref{use fine result} and \eqref{third expression}, we obtain
\begin{equation}\label{beframfina}
\sum_{m=1}^{\infty}  z^{m} q^{2 m - 1 }  \frac{(q;q^2)_{m-1}}{(zq;q^2)_{m}}= \sum_{k=1}^{\infty} \frac{(q^2; q^2)_{k-1} }{(z q^2; q^2)_{k}} (z q)^{k}.
\end{equation}
By a result of Ramanujan \cite[Entry 1.7.2, p.~29]{abramlostII}, if $|b|<1$ and $a$ is an arbitrary complex number, then 
\begin{align}\label{ramfina}
\sum_{n=0}^{\infty} \frac{(-1)^n (-q;q)_n (- \frac{aq}{b} ; q)_n \, b^n }{(aq; q^2)_{n+1} }  = \sum_{n=0}^{\infty}  \frac{(-1)^n (- \frac{aq}{b} ; q)_n \, b^n q^{\frac{n(n+1)}{2}} }{(-b; q)_{n+1}}.
\end{align}
Replace $a$ by $z q$ and $b$ by $-zq$ in \eqref{ramfina} to get
\begin{align}\label{aftramfina}
\sum_{n=1}^{\infty} \frac{(q^2; q^2)_{n-1} }{(z q^2; q^2)_{n}} (z q)^{n} =  \sum_{n=1}^{\infty} \frac{z^n q^{ \frac{n(n+1)}{2}} (q;q)_{n-1} }{ (zq; q)_n  }.
\end{align}
From  \eqref{new identity}, \eqref{beframfina} and \eqref{aftramfina}, we obtain \eqref{garvan's 2nd thm}. 
\begin{remark}
Let $z=1$ in \eqref{new identity}, \eqref{beframfina} and \eqref{aftramfina}. This gives
\begin{align*}
\sum_{n=1}^{\infty}  \frac{(-1)^{n-1} q^{n^2} }{(q;q^2)_{n} (1-q^{2n}) } = \sum_{n=1}^{\infty}  \frac{q^{2 n - 1 }}{1-q^{2n-1}}=\sum_{n=1}^{\infty}\frac{q^n}{1-q^{2n}}=\sum_{n=1}^{\infty}\frac{q^{n(n+1)/2}}{1-q^n}.
\end{align*}
The equalities between the last three expressions are well-known; see, for example, \cite[p.~28]{macmahon}.
\end{remark}
Now if we let $z=-1$ in \eqref{new identity}, \eqref{beframfina} and \eqref{aftramfina}, we get
\begin{align}\label{garvandiv}
&-\sum_{n=1}^{\infty}  \frac{q^{n^2} }{(-q;q^2)_{n} (1+ q^{2n}) }= \sum_{n=1}^{\infty}  (-1)^{n} q^{2 n - 1 }  \frac{(q;q^2)_{n-1}}{(-q;q^2)_{n}}=\sum_{n=1}^{\infty} \frac{(q^2; q^2)_{n-1} }{(-q^2; q^2)_{n}} (-q)^{n}\nonumber\\
&=\sum_{n=1}^{\infty} \frac{(-1)^n q^{ \frac{n(n+1)}{2}} (q;q)_{n-1} }{ (-q; q)_n  }.
\end{align}
Equation \eqref{garvandiv} encodes interesting number-theoretic and partition-theoretic information. To see this, let $\mathcal{Q}(n)$ be the set of partitions of $n$ in which all parts except possibly the largest part are odd and all odd positive integers less than or equal to the largest part occur as parts. Also let $\ell_{O}(\pi)$ denote the largest odd part in a partition. Then Garvan \cite[Corollary 1.3(ii)]{garvan1} has obtained the weighted partition identity resulting from equating the extreme sides of \eqref{garvandiv}, which can be stated as
\begin{equation}\label{gar1}
\sum_{\pi\in\mathcal{Q}(n)}(-1)^{\frac{\ell_{O}(\pi)+1}{2}+\#(\pi)}=(-1)^{\frac{n(n-1)}{2}}d_{8,1}(n),
\end{equation}
where $d_{8, 1}(n)$ is the number of divisors of $n$ congruent to $\pm 1\pmod{8}$ minus the number of divisors of $n$ congruent to $\pm 3\pmod{8}$.

With the two new expressions linking the extreme sides in \eqref{garvandiv} that we have obtained, further information can be extracted, namely, if $\mathcal{P}_{\text{odd}}(n)$ denotes the number of partitions of $n$ into odd parts, then the expressions on the extreme sides of \eqref{garvandiv} are also equal to 
\begin{equation}\label{gar2}
\sum_{\pi\in\mathcal{P}_{\text{odd}}(n)}2^{\nu_{d}(\pi)-1}(-1)^{\frac{l(\pi)+1}{2}+\#(\pi)}
\end{equation}
since
\begin{equation*}
\sum_{n=1}^{\infty}  (-1)^{n} q^{2 n - 1 }  \frac{(q;q^2)_{n-1}}{(-q;q^2)_{n}}=\sum_{n=1}^{\infty}\bigg(\sum_{\pi\in\mathcal{P}_{\text{odd}}(n)}(-1)^{\frac{l(\pi)+1}{2}+\#(\pi)}2^{\nu_{d}(\pi)-1}\bigg)q^n,
\end{equation*}
as can be seen using similar techniques employed in Section \ref{weight}. 

Moreover, we note that the third expression in \eqref{garvandiv} occurs in a recent work of Wang and Yee \cite{wangyee}. From their Theorem 1.3 and Corollary 5.2, we have
\begin{equation*}
\sum_{n=1}^{\infty} \frac{(q^2; q^2)_{n-1} }{(-q^2; q^2)_{n}} (-q)^{n}=\sum_{k=1}^{\infty}\sum_{m=-\big\lfloor\frac{k}{2}\big\rfloor}^{\big\lfloor\frac{k-1}{2}\big\rfloor}(-1)^{m-1}q^{k^2-2m^2}.
\end{equation*}
Thus, if 
\begin{equation*}
r^{*}(n):=\sum_{k^2-2m^2=n\atop k\geq 1, -\big\lfloor\frac{k}{2}\big\rfloor\leq m\leq\big\lfloor\frac{k-1}{2}\big\rfloor}(-1)^{m-1},
\end{equation*}
then the expressions in \eqref{gar1} and \eqref{gar2} are also equal to $r^{*}(n)$.

\subsection{Another proof of a result of Andrews, Garvan and Liang}
As we have seen, identity \eqref{ffwz} follows from \eqref{gen of Garavan} by letting $z\to 1$. Another proof of it can be obtained through the Bhargava-Adiga summation \cite[Equation (1.1)]{bhargavaadigaitsf} (see also \cite[Equation (1.3)]{andrewsba}), which states that for $|a|<1, |d|<1$ and $b\neq q^m$, $m\in\mathbb{Z}$,
\begin{equation}\label{ba}
\sum_{n=-\infty}^{\infty}\frac{(q/a)_na^n}{(d)_n(1-bq^n)}=\frac{(d/b)_{\infty}(ab)_{\infty}(q)_{\infty}^{2}}{(q/b)_{\infty}(d)_{\infty}(a)_{\infty}(b)_{\infty}}.
\end{equation}
Replace $d$ by $q$ and $b$ by $c$ in \eqref{ba} to get
\begin{align*}
\sum_{n=-\infty}^{\infty} \frac{\left( \frac{q}{a} \right)_n \, a^n  }{(q)_n (1- c q^n) } = \frac{ ( ac)_\infty (q)_\infty }{ (  c)_\infty (a)_\infty }.
\end{align*}
Now split the sum on the left into two sums, one over $n\geq 0$ and the other over $n<0$. Since $1/(q)_n=0$ for $n<0$, the second sum vanishes. Separate the $n=0$ term to get
\begin{align*}
\frac{1}{1-c} + \sum_{n=1}^{\infty} \frac{\left( \frac{q}{a} \right)_n \, a^n  }{(q)_n (1- c q^n) } =  \frac{ ( ac)_\infty (q)_\infty }{ (  c)_\infty (a)_\infty }. 
\end{align*} 
Now let $a\to 0$ to deduce \eqref{ffwz} upon simplification.
\begin{remark}
Let $c=q$ in \eqref{ffwz} and simplify so as to obtain
\begin{align*}
\sum_{n=1}^{\infty} \frac{(-1)^{n-1}  q^{\frac{n(n+1)}{2} } }{(1- q^{n+1}) (q)_n  } = \frac{q}{1-q}. 
\end{align*}
If we let $a\to 0$, $b=zq$ and $c=q$ in \eqref{entry3gen00} and then let $z\to 1$, we obtain
\begin{align*}
\sum_{n=1}^{\infty} \frac{(-1)^{n-1}  q^{\frac{n(n+1)}{2} } }{(1- q^{n+1}) (q)_n  }=\lim_{z\to 1}(1-z)\sum_{n=0}^{\infty} z^n    \sum_{m=1}^{\infty} \frac{z^m q^m }{1- q^{m+n+1}  }.
\end{align*}
Comparing the right-hand sides of the above two equations, we see that
the double series $\displaystyle\sum_{n=0}^{\infty} \sum_{m=1}^{\infty} \frac{q^m }{1- q^{m+n+1}  }$ diverges.
\end{remark}

\section{Important $q$-series identities through differentiation}\label{diffe}
Here we first prove Theorem \ref{genc} and then proceed on deriving Corollaries \ref{spt} and \ref{nsc}.
\begin{proof}[Theorem \textup{\ref{genc}}][]
Differentiate both sides of \eqref{gen of Garavan} with respect to $z$ so as to obtain
\begin{align*}
&\sum_{n=1}^{\infty}\frac{(-1)^{n-1}z^{n-1}q^{n(n+1)/2}}{(1-cq^n)(zq)_n}\left(n+\sum_{k=1}^{n}\frac{zq^k}{1-zq^k}\right)\nonumber\\
&=\frac{1}{c}\sum_{n=1}^{\infty}\frac{(\frac{zq}{c})_{n-1}(cq)^n}{(zq)_n}+\frac{z}{c}\sum_{n=1}^{\infty}\frac{(\frac{zq}{c})_{n-1}(cq)^n}{(zq)_n}\left(\sum_{k=1}^{n-1}\frac{-q^k/c}{1-zq^k/c}+\sum_{k=1}^{n}\frac{q^k}{1-zq^k}\right).
\end{align*}
Now let $z=1$ and employ \eqref{q-binomial thm} to simplify the first series on the right-hand side so that
\begin{align}\label{gencsimp}
&\sum_{n=1}^{\infty}\frac{(-1)^{n-1}nq^{n(n+1)/2}}{(1-cq^n)(q)_n}+\sum_{n=1}^{\infty}\frac{(-1)^{n-1}q^{n(n+1)/2}}{(1-cq^n)(q)_n}\sum_{k=1}^{n}\frac{q^k}{1-q^k}\nonumber\\
&=\frac{1}{c-1}\left(\frac{(q)_{\infty}}{(cq)_{\infty}}-1\right)+\frac{1}{c}\sum_{n=1}^{\infty}\frac{(\frac{q}{c})_{n-1}(cq)^n}{(q)_n}\left(\sum_{k=1}^{n}\frac{q^k}{1-q^k}-\sum_{k=1}^{n-1}\frac{q^k/c}{1-q^k/c}\right).
\end{align}
To simplify the second series on the left, we employ van Hamme's identity \cite{hamme}
\begin{align}\label{hammeid}
\sum_{k=1}^{n} \frac{q^k}{1 - q^k} = \sum_{k=1}^{n}\left[\begin{matrix} n\\k\end{matrix}\right]\frac{(-1)^{k-1} q^{k(k+1)/2}}{(1-q^k)},
\end{align}
which gives
\begin{align}\label{fottbef}
\sum_{n=1}^{\infty} \frac{(-1)^{n-1}  q^{n(n+1)/2}}{(1-cq^n) (q)_n } \sum_{k=1}^n \frac{q^k}{1 - q^k }&=\sum_{n=1}^{\infty} \frac{(-1)^{n-1} q^{n(n+1)/2}}{(1-cq^n) } \sum_{k=1}^{n} \frac{(-1)^{k-1} q^{k(k+1)/2}}{(1-q^k) (q)_k (q)_{n-k}}\nonumber\\
&= \sum_{m=0}^{\infty} \sum_{k=1}^{\infty} \frac{(-1)^{m} q^{\frac{(m+k)(m+k+1)}{2}} }{1 - cq^{m+k}} \frac{q^{k(k+1)/2}}{(q)_k (q)_m (1-q^k) } \nonumber\\
&=\sum_{k=1}^{\infty} \frac{q^{k^2 + k}}{(q)_k (1- q^k)} \sum_{m=0}^{\infty}  \frac{ (-1)^{m} q^{\frac{m(m+1)}{2} + m k } }{ (q)_m  (1-cq^{m+k})  }\nonumber \\
&=\sum_{k=1}^{\infty} \frac{q^{k^2 + k} (q^{k+1})_\infty }{(q)_k (1- q^k)} F(0, q^k; cq^k),
\end{align}
where in the last step, we used the $a\to 0$ case of \eqref{fine16.3}.

To simplify the series on the right-hand side of \eqref{gencsimp}, we use the result of Guo and Zhang \cite[Corollary 3.1]{guozhang} which states that if $n\geq 0$ and $0\leq m\leq n$,
\begin{align*}
&\sum_{k=0\atop k\neq m}^{n}\left[\begin{matrix} n\\k\end{matrix}\right]\frac{(q/z)_k(z)_{n-k}}{1-q^{k-m}}z^k\nonumber\\
&=(-1)^mq^{\frac{m(m+1)}{2}}\left[\begin{matrix} n\\m\end{matrix}\right](zq^{-m})_n\left(\sum_{k=0}^{n-1}\frac{zq^{k-m}}{1-zq^{k-m}}-\sum_{k=0\atop k\neq m}^{n}\frac{q^{k-m}}{1-q^{k-m}}\right).
\end{align*}
Let $m=0$ in the above identity and simplify to obtain
\begin{align}\label{m0}
\sum_{k=1}^{n}\frac{q^k}{1-q^k}-\sum_{k=1}^{n-1}\frac{zq^k}{1-zq^k}=\frac{z}{1-z}-\frac{1}{(z)_n}\sum_{k=1}^{n}\left[\begin{matrix} n\\k\end{matrix}\right]\frac{(q/z)_k(z)_{n-k}z^k}{1-q^k}.
\end{align}
That this is a generalization of \eqref{hammeid} is easily seen by letting $z\to 0$ on both sides. 

Invoking \eqref{m0} with $z=1/c$ in the series on the right-hand side of \eqref{gencsimp}, we see that
\begin{align}\label{fottbefser}
&\frac{1}{c}\sum_{n=1}^{\infty}\frac{(\frac{q}{c})_{n-1}(cq)^n}{(q)_n}\left(\sum_{k=1}^{n}\frac{q^k}{1-q^k}-\sum_{k=1}^{n-1}\frac{q^k/c}{1-q^k/c}\right)\nonumber\\
&=\frac{1}{(c-1)^2}\left(\frac{(q)_{\infty}}{(cq)_{\infty}}-1\right)-\frac{1}{c-1}\sum_{n=1}^{\infty}\frac{(cq)^n}{(q)_n}\sum_{k=1}^{n}\left[\begin{matrix} n\\k\end{matrix}\right]\frac{(cq)_k(1/c)_{n-k}c^{-k}}{1-q^k}\nonumber\\
&=\frac{1}{(c-1)^2}\left(\frac{(q)_{\infty}}{(cq)_{\infty}}-1\right)-\frac{1}{c-1}\sum_{j=0}^{\infty}\frac{(1/c)_j(cq)^j}{(q)_j}\sum_{k=1}^{\infty}\frac{(cq)_k}{(q)_k}\frac{q^k}{1-q^k}.
\end{align}
Finally, substitute \eqref{fottbef} and \eqref{fottbefser} in \eqref{gencsimp} so as to obtain \eqref{finefeqn} upon simplification.
\end{proof}
We now give various corollaries that follow from Theorem \ref{genc}.

It is first shown that Andrews' famous identity for the $\textup{spt}$-function, namely \eqref{sptidentity}, can be derived from Theorem \ref{genc}. This identity was the main ingredient in obtaining his beautiful congruences for $\textup{spt}(n)$ modulo $5, 7$ and $13$ \cite[Theorem 2]{andrews08}. The generating function version \cite[Theorem 4, p. 137]{andrews08} is
\begin{align*}
\sum_{n=1}^{\infty} \mathrm{spt}(n) q^n = \frac{1}{(q)_\infty} \sum_{n=1}^{\infty} \frac{n q^n }{1 -q^n } + \frac{1}{(q)_\infty} \sum_{n=1}^{\infty} \frac{(-1)^n q^{n(3n+1)/2}(1+q^n)}{(1-q^n)^2}.
\end{align*}
Note that \cite[Equations (3.3), (3.4)]{andrews08}
\begin{align}
\frac{1}{(q)_\infty} \sum_{n=1}^{\infty} \frac{n q^n }{1 -q^n }&=\sum_{n=1}^{\infty}np(n)q^n\label{npn}\\
\frac{1}{(q)_\infty} \sum_{n=1}^{\infty} \frac{(-1)^n q^{n(3n+1)/2}(1+q^n)}{(1-q^n)^2}&=\sum_{n=0}^{\infty}\frac{-1}{2}N_{2}(n)q^n.\label{n2ne}
\end{align}
Andrews proved \eqref{n2ne} by applying the operator $\left.\frac{d^2}{dz^2}\right|_{z=1}$ to a representation of the generating function of $N(m, n)$, the number of partitions of $n$ with rank $m$, namely \cite[Equations (2.15), (2.17)]{andrews08}
\begin{equation}\label{nmngf1}
R(z;q):=1+\sum_{m=-\infty}^{\infty}\sum_{n=1}^{\infty}N(m, n)z^mq^n=\frac{(1-z)}{(q)_{\infty}}\sum_{n=-\infty}^{\infty}\frac{(-1)^nq^{n(3n+1)/2}}{1-zq^n}.
\end{equation}
Here we require another representation of $R(z; q)$, that is \cite[Equation (2.16)]{andrews08},
\begin{equation}\label{nmngf2}
R(z;q)=\sum_{n=0}^{\infty}\frac{q^{n^2}}{(zq)_n(z^{-1}q)_n}.
\end{equation}
\begin{proof}[Corollary \textup{\ref{spt}}][]
Let $c\to 1$ on both sides of \eqref{finefeqn} to get
\begin{align}\label{finefeqnc1}
&\sum_{n=1}^{\infty}\frac{(-1)^{n-1}nq^{n(n+1)/2}}{(1-q^n)(q)_n}+\sum_{n=1}^{\infty}\frac{q^{n(n+1)}(q^{n+1})_{\infty}}{(1-q^n)(q)_n}F(0,q^n;q^n)\nonumber\\
&=\lim_{c\to 1}\left\{\frac{-c}{(1-c)^2}+\frac{(q)_{\infty}}{(c)_{\infty}}\left(\frac{c}{1-c}+\sum_{n=1}^{\infty}\frac{(cq)_n}{(q)_n}\frac{q^n}{1-q^n}\right)\right\}.
\end{align}
From \cite[Theorem 3.8, Equation (3.24)]{agl13},
\begin{align}\label{Andrews-Garvan-Liang}
\sum_{n=1}^{\infty} \mathrm{spt}(n) q^n = \frac{1}{(q)_\infty }\sum_{n=1}^{\infty} \frac{(-1)^{n-1} n q^{n(n+1)/2}}{(1-q^n) (q)_n }.
\end{align}
We first show that
\begin{align}\label{befvan}
\sum_{n=1}^{\infty}\frac{q^{n(n+1)}(q^{n+1})_{\infty}}{(1-q^n)(q)_n}F(0,q^n;q^n)=(q)_{\infty}\sum_{n=0}^{\infty}\frac{1}{2}N_{2}(n)q^n.
\end{align}
From \cite[p.~13, Equation (12.33)]{fine},
\begin{align}\label{fott}
F(0,t;t) = \frac{1}{1-t} \sum_{j=0}^{\infty} \frac{t^{2j} q^{j^2} }{(t q)_j^2}.
\end{align}
Invoking \eqref{fott} with $t=q^k$ in \eqref{fottbef}, we see that
\allowdisplaybreaks{\begin{align}\label{fottaft}
\sum_{n=1}^{\infty}\frac{q^{n(n+1)}(q^{n+1})_{\infty}}{(1-q^n)(q)_n}F(0,q^n;q^n)&= (q)_{\infty} \sum_{n=1}^{\infty}   \frac{q^{n^2 + n} }{(q)_n^2 (1- q^n)^2} \sum_{j=0}^{\infty} \frac{q^{j^2 +2n j}}{(q^{n+1})_j^2} \nonumber \\
& = (q)_{\infty} \sum_{n=1}^{\infty}   \frac{q^{ n} }{ (1- q^n)^2} \sum_{j=n}^{\infty} \frac{q^{j^2} }{(q)_j^2} \nonumber \\
 & = (q)_{\infty} \sum_{j=1}^{\infty} \frac{q^{j^2} }{(q)_j^2}    \sum_{n=1}^{j} \frac{q^{ n} }{ (1- q^n)^2}\nonumber\\
&= \frac{1}{2}(q)_{\infty}\left. \frac{\mathrm{d}^2}{\mathrm{d}z^2}  \sum_{n=0}^{\infty} \frac{q^{n^2}}{(z q)_n (z^{-1} q)_n }\right|_{z=1},
\end{align}}
as can be seen from a laborious but straightforward calculation. From \eqref{nmngf2},
\begin{align}\label{fottaft1}
\frac{1}{2}(q)_{\infty}\left. \frac{\mathrm{d}^2}{\mathrm{d}z^2}  \sum_{n=0}^{\infty} \frac{q^{n^2}}{(z q)_n (z^{-1} q)_n }\right|_{z=1}&=\frac{1}{2}(q)_{\infty}\left. \frac{\mathrm{d}^2}{\mathrm{d}z^2} R(z;q)\right|_{z=1}\nonumber\\
&=\frac{1}{2}(q)_{\infty}\sum_{n=0}^{\infty}N_{2}(n)q^n,
\end{align}
where the last step follows easily from the first equality in \eqref{nmngf1} or from \cite[Equation (3.4)]{andrews08}. Equation \eqref{befvan} now follows from \eqref{fottaft} and \eqref{fottaft1}. 

Lastly, we show that
\begin{align}\label{limic}
\lim_{c\to1}\left\{\frac{-c}{(1-c)^2}+\frac{(q)_{\infty}}{(c)_{\infty}}\left(\frac{c}{1-c}+\sum_{n=1}^{\infty}\frac{(cq)_n}{(q)_n}\frac{q^n}{1-q^n}\right)\right\}=\sum_{n=1}^{\infty}\frac{q^n}{1-q^n}.
\end{align}
Let $L$ denote the above limit. Then employing \eqref{q-binomial thm} in the second step below, we see that
\begin{align}\label{limic1}
L&=\lim_{c\to1}\frac{(q)_{\infty}}{(cq)_{\infty}}\lim_{c\to 1}\frac{1}{1-c}\left\{\frac{c}{1-c}-\frac{c}{1-c}\sum_{n=0}^{\infty}\frac{(c)_n}{(q)_n}q^n+\sum_{n=1}^{\infty}\frac{(cq)_n}{(q)_n}\frac{q^n}{1-q^n}\right\}\nonumber\\
&=\lim_{c\to1}\frac{1}{1-c}\left\{-c\sum_{n=1}^{\infty}\frac{(cq)_{n-1}}{(q)_n}q^n+\sum_{n=1}^{\infty}\frac{(cq)_n}{(q)_n}\frac{q^n}{1-q^n}\right\}\nonumber\\
&=\lim_{c\to1}\frac{1}{1-c}\sum_{n=1}^{\infty}\frac{(cq)_{n-1}}{(q)_n}q^n\left(-c+\frac{1-cq^n}{1-q^n}\right)\nonumber\\
&=\sum_{n=1}^{\infty}\frac{q^n}{(1-q^n)^2}\nonumber\\
&=\sum_{n=1}^{\infty}\frac{nq^n}{1-q^n}.
\end{align}
From \eqref{finefeqnc1}, \eqref{npn}, \eqref{Andrews-Garvan-Liang}, \eqref{befvan} and \eqref{limic}, we arrive at \eqref{sptidentity}.
\end{proof}
\begin{remark}
Note that one can equivalently start from Ramanujan's identity \eqref{Garvan's identity}, differentiate both sides with respect to $z$ and then let $z=1$ to arrive at \eqref{sptidentity}. This way, the calculation involving the right-hand side of \eqref{Garvan's identity} would be straightforward in comparison with that in \eqref{limic1}. However, our intention to derive \eqref{sptidentity} from \eqref{finefeqn} is to show the uniformity in the approach in deriving \eqref{sptidentity}, \eqref{c0eqn} and \eqref{nsceqn}.
\end{remark}
\begin{proof}[Corollary \textup{\ref{c0}}][]
Let $c=0$ in Theorem \ref{genc} and divide both sides of the resulting identity by $(q)_{\infty}$. This gives
\begin{align}\label{befmer}
\frac{1}{(q)_{\infty}}\sum_{n=1}^{\infty}\frac{(-1)^{n-1}nq^{n(n+1)/2}}{(q)_n}+\sum_{n=1}^{\infty}\frac{q^{n(n+1)}}{(1-q^n)(q)_n^2}=\sum_{n=1}^{\infty}\frac{q^n}{(1-q^n)(q)_n}.
\end{align}
By a recent result of Merca \cite[Theorem 1]{merca},
\begin{equation}\label{mercaeqn}
\frac{1}{(q)_{\infty}}\sum_{n=1}^{\infty}\frac{(-1)^{n-1}nq^{n(n+1)/2}}{(q)_n}=\sum_{n=1}^{\infty}\frac{q^n}{1-q^n}.
\end{equation}
Substituting \eqref{mercaeqn} in \eqref{befmer}, we arrive at \eqref{c0eqn}.

In order to prove \eqref{dwlpt}, we first observe that
\begin{align}\label{lptcal}
\sum_{n=1}^{\infty}\frac{q^n}{(1-q^n)(q)_n}=\sum_{n=1}^{\infty}\frac{1}{(q)_{n-1}}\frac{q^n}{(1-q^n)^{2}}=\sum_{n=1}^{\infty}\textup{lpt}(n)q^n.
\end{align}
In order to show that
\begin{equation*}
\sum_{n=1}^{\infty}\frac{q^{n(n+1)}}{(1-q^n)(q)_n^2}=\sum_{n=1}^{\infty}w(n)q^n,
\end{equation*}
where $w(n)$ is defined in \eqref{wn}, we write
\begin{align}\label{wncal}
&\sum_{n=1}^{\infty}\frac{q^{n(n+1)}}{(1-q^n)(q)_n^2}\nonumber\\
&=\sum_{n=1}^{\infty}\left(\frac{q}{1-q}\right)^2\cdots\left(\frac{q^{n-1}}{1-q^{n-1}}\right)^2\cdot\frac{q^{2n}}{(1-q^n)^3}\nonumber\\
&=\sum_{n=1}^{\infty}\left(\sum_{k_1=2}^{\infty}(k_1-1)q^{k_1}\right)\cdots\left(\sum_{k_{n-1}=2}^{\infty}(k_{n-1}-1)q^{(n-1)k_{n-1}}\right)\left(\sum_{k_n=2}^{\infty}\frac{1}{2}k_n(k_n-1)q^{nk_n}\right)\nonumber\\
&=\sum_{n=1}^{\infty}\sum_{k_i\geq 2,1\leq i\leq n}^{\infty}\frac{1}{2}k_n(k_n-1)\prod_{i=1}^{n-1}(k_i-1)q^{k_1+2k_2+\cdots+nk_n}\nonumber\\
&=\sum_{n=1}^{\infty}w(n)q^n.
\end{align}

From \eqref{c0eqn}, \eqref{lptcal} and \eqref{wncal}, we deduce \eqref{dwlpt}.
\end{proof}
\begin{proof}[Corollary \textup{\ref{wdpn}}][]
We begin with \footnote{It was during the international conference on the occasion of Ramanujan's 125th birth anniversary at University of Delhi in December 2012 that the first author learned from George E. Andrews about the connection between $\textup{lpt}(n)$ and Zagier's identity in Equation \eqref{and}. Both the authors thank Andrews for the same.}
\allowdisplaybreaks{\begin{align}\label{and}
\sum_{n=1}^{\infty}\textup{lpt}(n)&=\sum_{n=1}^{\infty}\frac{q^n}{(1-q^n)(q)_n}\nonumber\\
&=\sum_{m, n=1}^{\infty}\frac{q^{mn}}{(q)_n}\nonumber\\
&=\sum_{m=1}^{\infty}\left(\frac{1}{(q^m)_{\infty}}-1\right)\nonumber\\
&=\frac{1}{2}-\sum_{n=1}^{\infty}\frac{q^n}{1-q^n}-\frac{1}{2(q)_{\infty}}\sum_{n=1}^{\infty}n\left(\frac{12}{n}\right)q^{\frac{n^2-1}{24}},
\end{align}}
where in the last step we applied \eqref{zagierid}. Combining this with \eqref{c0eqn}, we see that
\begin{equation}\label{preroot}
\sum_{n=1}^{\infty}\frac{q^{n(n+1)}}{(1-q^n)(q)_n^2}=\frac{1}{2}-2\sum_{n=1}^{\infty}\frac{q^n}{1-q^n}-\frac{1}{2(q)_{\infty}}\sum_{n=1}^{\infty}n\left(\frac{12}{n}\right)q^{\frac{n^2-1}{24}}.
\end{equation}
Now
\begin{align}\label{root}
\frac{1}{(q)_{\infty}}\sum_{n=1}^{\infty}n\left(\frac{12}{n}\right)q^{\frac{n^2-1}{24}}=\sum_{n=1}^{\infty}\left(\sum_{k=1}^{\left\lfloor\sqrt{24n+1}\right\rfloor}k\left(\frac{12}{k}\right)p\left(n-\frac{(k^2-1)}{24}\right)\right)q^n.
\end{align}
Thus comparing the coefficients of $q^n$ on both sides of \eqref{preroot} and employing \eqref{wncal} and \eqref{root}, we get \eqref{wnrep}.
\end{proof}
\begin{proof}[Corollary \textup{\ref{nsc}}][]
Let $c=-1$ in \eqref{finefeqn} and use
\cite[Theorem 3.8, Equation (3.25)]{agl13},
\begin{equation*}
\sum_{n=1}^{\infty}\frac{(-1)^{n-1}nq^{n(n+1)/2}}{(q)_n(1+q^n)}=(q)_{\infty}\sum_{n=1}^{\infty}N_{\textup{SC}}(n)q^n,
\end{equation*}
to get
\begin{align}\label{nscagl1}
&(q)_{\infty}\sum_{n=1}^{\infty}N_{\textup{SC}}(n)q^n+\sum_{n=1}^{\infty}\frac{q^{n(n+1)}(q^{n+1})_{\infty}}{(1-q^n)(q)_n}F(0,q^n;-q^n)\nonumber\\
&=\frac{1}{4}-\frac{1}{4}\frac{(q)_{\infty}}{(-q)_{\infty}}+\frac{1}{2}\frac{(q)_{\infty}}{(-q)_{\infty}}\sum_{n=1}^{\infty}\frac{(-q)_n}{(q)_n}\frac{q^n}{1-q^n}.
\end{align}
\end{proof}
We now simplify the second series on the left using \cite[p.~17, Equation (15.51)]{fine},
\begin{equation*}
F(0,q^n;-q^n)=\frac{1-q^n}{(-q^n)_{\infty}}\sum_{j=0}^{\infty}\frac{q^{nj+j(j+1)/2}}{(1-q^{n+j})(q)_j}.
\end{equation*}
Thus
\allowdisplaybreaks{\begin{align}\label{aq}
\sum_{n=1}^{\infty}\frac{q^{n(n+1)}(q^{n+1})_{\infty}}{(1-q^n)(q)_n}F(0,q^n;-q^n)&=\sum_{n=1}^{\infty}\frac{q^{n(n+1)}}{(q)_n}\frac{(q^{n+1})_{\infty}}{(-q^n)_{\infty}}\sum_{j=0}^{\infty}\frac{q^{nj+j(j+1)/2}}{(1-q^{n+j})(q)_j}\nonumber\\
&=\frac{1}{2}\frac{(q)_{\infty}}{(-q)_{\infty}}\sum_{n=1}^{\infty}\frac{q^{n(n+1)/2}(-1)_n}{(q)_n^{2}}\sum_{j=0}^{\infty}\frac{q^{(n+j)(n+j+1)/2}}{(q)_j(1-q^{n+j})}\nonumber\\
&=\frac{1}{2}\frac{(q)_{\infty}}{(-q)_{\infty}}\sum_{k=1}^{\infty}\frac{q^{k(k+1)/2}}{(q)_k(1-q^k)}\sum_{n=1}^{k}\bigg[\begin{matrix} k\\n\end{matrix}\bigg]\frac{(-1)_n}{(q)_n}q^{n(n+1)/2}\nonumber\\
&=\frac{1}{2}\frac{(q)_{\infty}}{(-q)_{\infty}}\sum_{k=1}^{\infty}\frac{q^{k(k+1)/2}}{(q)_k(1-q^k)}\left(\frac{(-q)_k}{(q)_k}-1\right),
\end{align}}
where the last step follows from the $q$-Chu-Vandermonde summation \cite[Corollary 1.2]{overp}. Finally, \eqref{nsceqn} follows from \eqref{nscagl1} and \eqref{aq}.

Before proving Corollary \ref{newdn}, we begin with a lemma.
\begin{lemma}\label{nscssptdlemma}
We have
\begin{align*}
2(-q)_{\infty}\sum_{n=1}^{\infty}N_{\textup{SC}}(n)q^n-\sum_{n=1}^{\infty}\frac{q^{\frac{n(n+1)}{2}}}{(1-q^n)(q)_n}=\sum_{n=1}^{\infty}\frac{q^n}{1-q^n}.
\end{align*}
\end{lemma} 
\begin{proof}[Corollary \textup{\ref{newdn}}][]
From \cite[Equation (1.12)]{agl13},
\begin{equation}\label{agl13nsc}
\sum_{n=1}^{\infty}N_{\textup{SC}}(n)q^n=\frac{1}{(-q)_{\infty}}\sum_{n=1}^{\infty}\frac{(-q)_{n-1}q^n}{1-q^n}.
\end{equation}
Let
\begin{equation}\label{defan}
\sum_{n=1}^{\infty}\frac{(-q)_{n-1}q^n}{1-q^n}=\sum_{n=1}^{\infty}a(n)q^n.
\end{equation}
Then $a(n)$ equals the number of partitions of $n$ in which only the largest part is allowed to repeat. From Andrews \cite[p.~153]{andrewsba}, these partitions are conjugates of compact partitions. Hence if $c(n)$ denotes the number of compact partitions of $n$, then $a(n)=c(n)$. However, by a result conjectured by Beck and recently proved by Chern \cite[Theorem 1.2]{chern}, $c(n)=\textup{ssptd}_o(n)$, where $\textup{ssptd}_o(n)$ denotes the sum of the smallest parts in all partitions of $n$ into distinct parts which are odd in number. Hence
\begin{equation}\label{assptd}
a(n)=\textup{ssptd}_o(n).
\end{equation}
Let $\textup{ssptd}(n)$ denote the sum of the smallest parts in all partitions of $n$ into distinct parts, and $\textup{ssptd}_e(n)$, the same with the added restriction that the parts be even in number. As can be conjured by observing that the left-hand side of \eqref{Uchimura} is the generating function of that of \eqref{ffwidty}, 
\begin{equation}\label{ssptdgen}
\sum_{n=1}^{\infty}\frac{q^{\frac{n(n+1)}{2}}}{(1-q^n)(q)_n}=\sum_{n=1}^{\infty}\textup{ssptd}(n)q^n=\sum_{n=1}^{\infty}\textup{ssptd}_o(n)q^n+\sum_{n=1}^{\infty}\textup{ssptd}_e(n)q^n.
\end{equation}
Thus, from \eqref{agl13nsc}, \eqref{defan}, \eqref{assptd} and \eqref{ssptdgen},
\begin{align*}
2(-q)_{\infty}\sum_{n=1}^{\infty}N_{\textup{SC}}(n)q^n-\sum_{n=1}^{\infty}\frac{q^{\frac{n(n+1)}{2}}}{(1-q^n)(q)_n}&=\sum_{n=1}^{\infty}\left(\textup{ssptd}_o(n)-\textup{ssptd}_e(n)\right)q^n\nonumber\\
&=\sum_{n=1}^{\infty}\frac{q^n}{1-q^n},
\end{align*}
where the last step follows from the Fokkink, Fokkink and Wang identity \eqref{ffwidty}. This proves the lemma.
\end{proof}
\begin{proof}[Corollary \textup{\ref{newdn}}][]
Multiply both sides of \eqref{nsceqn} by $2(-q)_{\infty}/(q)_{\infty}$, employ Lemma \ref{nscssptdlemma} and simplify to obtain
\begin{align}\label{preoverpgen}
\sum_{n=1}^{\infty}\frac{(-q)_n}{(q)_n}\frac{q^n}{1-q^n}-\sum_{n=1}^{\infty}\frac{(-q)_{n}}{(q)_n^{2}}\frac{q^{n(n+1)/2}}{1-q^n}+\frac{1}{2}\frac{(-q)_{\infty}}{(q)_{\infty}}-\frac{1}{2}=\sum_{n=1}^{\infty}\frac{q^n}{1-q^n}.
\end{align}
Now from \cite[Equation (1.3)]{overp},
\begin{equation}\label{overpgen}
\frac{(-q)_{\infty}}{(q)_{\infty}}=\sum_{n=0}^{\infty}\frac{(-1)_n}{(q)_{n}^{2}}q^{n(n+1)/2}.
\end{equation}
Invoke \eqref{overpgen} in \eqref{preoverpgen} and simplify to obtain \eqref{newdneqn}.

To prove \eqref{newdnrep}, note that
\begin{align}\label{newdnrep1}
\sum_{n=1}^{\infty}\frac{(-q)_{n-1}}{(q)_n^{2}}\frac{q^{n(n+3)/2}}{1-q^n}=\sum_{n=1}^{\infty}\bigg(\sum_{\pi\in\mathcal{P}^{*}(n)\atop L(\pi)\geq 2}\frac{L(\pi)(L(\pi)-1)}{2}\prod_{i=1}^{l(\pi)-1}(2\nu(i)-1)\bigg)q^n
\end{align}
can be proved along exact similar lines as \eqref{wncal}. As far as proving 
\begin{align}\label{newdnrep2}
\sum_{n=1}^{\infty}\frac{(-q)_n}{(q)_n}\frac{q^n}{1-q^n}=\sum_{n=1}^{\infty}\bigg(\sum_{\pi\in\mathcal{P}_{o}(n) \atop l(\pi)\text{overlined}}(2L(\pi)-1)\bigg)q^n
\end{align}
is concerned, it is easily seen that
\begin{align*}
\sum_{n=1}^{\infty}\frac{(-q)_n}{(q)_n}\frac{q^n}{1-q^n}&=\sum_{n=1}^{\infty}\frac{(-q)_{n-1}}{(q)_{n-1}}\frac{q^n(1+q^n)}{(1-q^n)^2}\nonumber\\
&=\sum_{n=1}^{\infty}\frac{(-q)_{n-1}}{(q)_{n-1}}\sum_{k_n=1}^{\infty}(2k_n-1)q^{nk_n},
\end{align*}
which proves \eqref{newdnrep2}. Together, \eqref{newdneqn}, \eqref{newdnrep1} and \eqref{newdnrep2} give \eqref{newdnrep}.
\end{proof}
Another special case of \eqref{finefeqn} when $c=q^m, m\geq 1$, is stated below without proof.
\begin{align*}
&\sum_{n=1}^{\infty}\frac{(-1)^{n-1}nq^{n(n+1)/2}}{(1-q^{n+m})(q)_n}+\sum_{n=1}^{\infty}\frac{q^{n(n+1)}(q^{n+1})_{\infty}}{(1-q^n)(q)_n}F(0,q^n;q^{n+m})\nonumber\\
&=\frac{q^m((q)_m-1)}{(1-q^m)^2}+(q)_{m-1}\sum_{n=1}^{\infty}\bigg[\begin{matrix} n+m\\n\end{matrix}\bigg]\frac{q^n}{1-q^n}.
\end{align*}
When $m=1$, it simplifies to
\begin{align}\label{ceqq}
&\sum_{n=1}^{\infty}\frac{(-1)^{n-1}nq^{n(n+1)/2}}{(1-q^{n+1})(q)_n}+\sum_{n=1}^{\infty}\frac{q^{n(n+1)}(q^{n+1})_{\infty}}{(1-q^n)(q)_n}F(0,q^n;q^{n+1})=\sum_{n=1}^{\infty}\frac{q^{n}}{1-q^{n}}
\end{align}

\section{Concluding remarks}\label{cr}
We hope to have demonstrated the richness of partition-theoretic information embedded in our generalizations \eqref{entry3gen}, \eqref{gen of Garavan} and \eqref{generalization of Entry 5} of Ramanujan's \eqref{neglected}, \eqref{Garvan's identity} and \eqref{fifth} respectively. 

In Corollary \ref{gen of sigma(q)}, we have shown that the one-variable generalization of $\sigma(q)$ defined in \eqref{otvg}, namely $\sigma(c, q)$, admits a simpler representation. We have also given the partition-theoretic meaning of its coefficients. Thus we may analogously ask if $\sigma(c, d, q)$, the two-variable generalization of $\sigma(q)$ defined in \eqref{otvg1} also admits a simpler representation than the complicated one given in \cite[Theorem 1.2]{bandix1}. However, following the approach in our paper to accomplish this would first require finding a generalization of \eqref{gen of Garavan} with one more variable $d$. Without doubt, this would be interesting in itself.

Except for the special case $c\to 1$ of Theorem \ref{genc}, that is, the one which gives Andrews' identity for $\textup{spt}(n)$, each of the cases $c=0, -1$ and $q$ in \eqref{c0eqn}, \eqref{newdneqn} and \eqref{ceqq} respectively involves the divisor generating function $\displaystyle\sum_{n=1}^{\infty}\frac{q^n}{1-q^n}$. This is intriguing, to say the least, and certainly merits further study. 

\begin{center}
\textbf{Acknowledgements}
\end{center}
The first author's research is supported by the SERB-DST grant ECR/2015/000070 whereas the second author is a SERB National Post Doctoral Fellow (NPDF) supported by the fellowship PDF/2017/000370. Both sincerely thank SERB-DST for the support.

\end{document}